\documentclass[a4paper,11pt]{article}

\usepackage{verbatim}

\usepackage[T1]{fontenc}
\usepackage{setspace}
\usepackage{indentfirst}

\usepackage{geometry}
\usepackage[hidelinks]{hyperref}

\usepackage{amsthm}
\usepackage{amsmath}
\usepackage{amsfonts}
\usepackage{amssymb}
\usepackage{mathdots}
\usepackage{mathtools}
\usepackage{bbm}
\usepackage{shuffle}

\usepackage{times}
\usepackage{bm}

\usepackage{tikz-cd}
\usepackage{tikz}
\usetikzlibrary{matrix,arrows,decorations.pathmorphing}
\usetikzlibrary{shapes.geometric}
\usepackage[usestackEOL]{stackengine}
\usepackage{fancybox}
\usepackage{caption}

\newcommand{\sbt}{\,\begin{picture}(-1,1)(-1,-3)\circle*{3}\end{picture}\ }

\usepackage[toc,page]{appendix}

\usepackage{todonotes}

\DeclareMathOperator{\Hom}{Hom}
\DeclareMathOperator{\rk}{rank}
\DeclareMathOperator{\id}{Id}
\DeclareMathOperator{\pr}{pr}
\DeclareMathOperator{\SL}{SL}
\DeclareMathOperator{\sign}{sgn}
\DeclareMathOperator{\st}{st}
\DeclareMathOperator{\Lie}{Lie}
\DeclareMathOperator{\mul}{\mathbf{Mult}}
\DeclareMathOperator{\GL}{GL}
\DeclareMathOperator{\Gr}{Gr}
\DeclareMathOperator{\diag}{diag} 
\newcommand{\hw}{\mathrm{hw}}

\usepackage{theoremref}
\newtheorem{Thm}{Theorem}[section]
\newtheorem{Main}[Thm]{Main Theorem}
\newtheorem{Pro}[Thm]{Proposition}
\newtheorem{Lem}[Thm]{Lemma}
\newtheorem{Cor}[Thm]{Corollary}
\newtheorem{Con}[Thm]{Conjecture}

\theoremstyle{definition}
\newtheorem{Def}[Thm]{Definition}
\newtheorem{Ex}[Thm]{Example}

\theoremstyle{remark}
\newtheorem{Rmk}[Thm]{Remark}

\newcommand\nnfootnote[1]{%
  \begin{NoHyper}
  \renewcommand\thefootnote{}\footnote{#1}%
  \addtocounter{footnote}{-1}%
  \end{NoHyper}
}

\usepackage{tikzsymbols}

\begin{document}
\title{Geometric Multiplicities}
\author{Arkady Berenstein \and Yanpeng Li}

\newcommand{\Addresses}{{
  \bigskip
  \footnotesize
  
  \textsc{Department of Mathematics, University of Oregon, Eugene, OR 97403, USA}\par\nopagebreak
  \textit{E-mail address}: \texttt{arkadiy@uoregon.edu}
  
  \medskip

  \textsc{Section of Mathematics, University of Geneva, 2-4 rue du Li\`evre, c.p. 64, 1211 Gen\`eve 4, Switzerland}\par\nopagebreak
  \textit{E-mail address}: \texttt{yanpeng.li@unige.ch}

}}
\date{}
\maketitle

\nnfootnote{\emph{Keywords:} Representations, Multiplicities, and Positivity}

\begin{abstract}
  In this paper, we introduce geometric multiplicities, which are  positive varieties with potential fibered over the Cartan subgroup $H$ of a reductive group $G$. They form a monoidal category and we construct a monoidal functor from this category to the representations of the Langlands dual group $G^\vee$ of $G$. Using this, we explicitly compute various multiplicities in $G^\vee$-modules in many ways. In particular,  we recover the formulas for tensor product multiplicities from \cite{BZ01} and generalize them in several directions. In the case when our geometric multiplicity $X$ is a monoid, {\em i.e.}, the corresponding $G^\vee$ module is an algebra,  we expect that in many cases, the spectrum of this algebra is affine $G^\vee$-variety $X^\vee$, and thus the correspondence $X\mapsto X^\vee$  has a flavor of both the Langlands duality and mirror symmetry.
\end{abstract}

\tableofcontents

\section{Introduction}

The goal of this paper is twofold: 
\begin{itemize}
  \item To continue and, to some extent, complete the ``multiplicity geometrization'' program, originated in \cite{BKI, BKII, BZ88,BZ92,BZ01}.

  \item To use our geometric multiplicities to recover all known and obtain many new formulas for such classical multiplicities as tensor product multiplicities, weight multiplicities, etc emerging in representation of complex reductive groups. 
\end{itemize}
In our approach, a {\em geometric multiplicity} is a positive variety with potential fibered over the Cartan subgroup $H$ of a reductive group $G$. They form a category, which we denote it by $\mul_G$ (see Definition \ref{mulg} for more details). 
\begin{Thm}[Theorem \ref{Associator}]
    The category is a non-strict (uniteless) monoidal category with product $M_1\star M_2$ given by
    \[
        M_1\star M_2:= M_1\times M_2\times U.
    \]
    (Here $U$ is the maximal unipotent subgroup of $G$ normalized by $H$.) 
\end{Thm}
The associator for $\mul_G$ is extremely non-trivial. We construct it via embedding $\mul_G$ into the monoidal category of decorated decorated $U\times U$-bicrystals, see Section \ref{section proofs1}. In fact, even though the category $\mul_G$ has no unit, it has a natural tropicalization functor $\mathcal{T}$ (see Proposition \ref{prop;equivofcats}) to what we call {\em affine tropical varieties} (see Definition \ref{afftrop}). The latter category is an ``honest'' monoidal category with the natural unite $0$ (thus, in what follows, we will ignore this minor deficiency of $\mul_G$).

Using this tropicalization functor $\mathcal{T}$, we can recover many interesting representation of the Langlands dual group $G^\vee$ of $G$ and various multiplicities in them out of geometric multiplicities over $G$ as follows. 

First, to any $M\in \mul_G$ we assign an affine tropical variety $M^t:=\mathcal{T}(M)$ fibered over the dominant coweight monoid $P^\vee_+$ of $G$. Then assign a $G^\vee$-module $\mathcal{V}(M)$ to $M$ via
\[
  \mathcal{V}(M)=\bigoplus_{\lambda^{\vee}\in P^{\vee}_+} \mathbb{C}[M^t_{\lambda^{\vee}}]\otimes V_{\lambda^{\vee}},
\]
where $M^t_{\lambda^\vee}$ is the tropical fiber over $\lambda^\vee$, $\mathbb{C}[\cdot]$ is the linearization of the set, and $V_{\lambda^\vee}$ is the irreducible representation of $G^\vee$ with highest weight $\lambda^\vee$  (thus we passed from the geometric multiplicities to the algebraic ones).
\begin{Thm}[Theorem \ref{functor}]
  The assignments $M\mapsto \mathcal{V}(M)$ define a monoidal functor $\mul_G$ to $\mathbf{Mod}_{G^\vee}$, the category of $G^\vee$-modules.
\end{Thm}

\begin{Rmk}
  It will be interesting to relate this result to geometric Satake, which asserts that the category of finite-dimensional $G^\vee$ modules is monoidal equivalent to the category of perverse sheaves on the affine Grassmannian $\Gr_G$ of $G$. Since our $G^\vee$-modules ${\mathcal V}(M)$ are  frequently $G^\vee$-module algebras (see discussion below), we expect  them, under geometric Satake, to give rise to interesting  perverse sheave on $\Gr_G$.
\end{Rmk}

Note, however, sometimes the module $\mathcal{V}(M)$ may have infinite multiplicities, in particular this happens for the module $\mathcal{V}(H\star H)=\mathcal{V}(H)\otimes V(H)$ because
\begin{equation}\label{intro:eq:1}
    \mathcal{V}(H)=\bigoplus_{\lambda^\vee\in P^\vee_+} V_{\lambda^\vee}.
\end{equation}
To address this issue, we decorate geometric multiplicities by additionally fibering them over another split algebraic torus. Also we amend the multiplication in $\mul_G$ in such a way that if $M_i$ is additionally fibered over $S_i$ for $i=1,2$, then $M_1\star M_2$ is additionally fibered over $S_1\times S_2 \times H$. This fixes the issue with $H\star H$ because now it is fibered over $H^{3}$. And the finiteness of the multiplicities is restored as follows. Given a geometric multiplicity $M$ additionally fibered over $S$, its tropicalization $M^t$ is naturally fibered over the direct product of $P^\vee_+$ and cocharacter lattice $X_*(S):=\Hom(\mathbb{G}_{\mathbf{m}}, S)$ so that for every cocharacter $\xi\in X_*(S)$ we define $\mathcal{V}_{\xi}(M)$ by
\[
  \mathcal{V}_{\xi}(M):=\bigoplus_{\lambda\in P^\vee_+} \mathbb{C}[M^t_{\lambda^\vee, \xi}]\otimes V_{\lambda^\vee}.
\]
\begin{Thm}[Theorem \ref{functor-compo}]
  Given geometric multiplicities $M_i$ additionally fibered over $S_i$ for $i=1,2$, one has the following natural isomorphism of $G^\vee$-modules
  \begin{equation}\label{eq:intro}
    \mathcal{V}_{\xi_1,\xi_2,\lambda^\vee, \nu^\vee}(M_1\star M_2) \cong I_{\lambda^\vee}\left(\mathcal{V}_{\xi_1}(M_1)\right)\otimes I_{\nu^\vee}\left(\mathcal{V}_{\xi_1}(M_2)\right),
  \end{equation}
  where $I_{\mu^\vee}(V)$ denotes the $\mu^\vee$-th isotypic component of a $G^\vee$-module $V$.
\end{Thm}
This indeed fixes the ``infinite multiplicity" issue for $\mathcal{V}(H\star H)$ since \eqref{eq:intro} boils down to an isomorphism
\[
  \mathcal{V}_{\lambda^\vee, \nu^\vee}(H\star H) \cong V_{\lambda^\vee}\otimes V_{\nu^\vee}.
\]
In turn,  by applying this argument repeatedly to the geometric multiplicities: 
\[
H^{\star n}:=\underbrace{H\star\cdots \star H}_n,
\]
we get an isomorphism of $G^\vee$-modules
\[
  \mathcal{V}_{\lambda_1^\vee,\ldots, \lambda_n^\vee}(H^{\star n}) \cong V_{\lambda^\vee_1}\otimes\cdots \otimes V_{\lambda^\vee_n},\quad \forall \lambda_1^\vee,\ldots, \lambda_n\in (P_+^\vee)^n, ~n\geqslant 2.
\]
Thus the geometric multiplicities $H^{\star n}$ (fibered over $H^{n+1}$)  compute all tensor product multiplicities $c_{\lambda_1^\vee,\ldots, \lambda_n^\vee}^{\mu^\vee}:=\dim \Hom_{G^\vee}(V_{\mu^\vee}, V_{\lambda_1^\vee}\otimes\cdots \otimes V_{\lambda_n^\vee})$. 

By modifying  $H\star H$, we construct the geometric multiplicity $X_{G,L}$ in the category $\mul_L$ for any Levi subgroup $L$ of $G$, which computes the all reduction multiplicities of $V_{\lambda^\vee}$ restricted to $L^\vee\subset G^\vee$, see Section \ref{reduction multi} for details.

\begin{Rmk}
  Our construction of $H^{\star n}$ bears some resemblance with the approach by Goncharov-Shen in \cite{GS1, GS2}. In particular, they related their configuration space $\text{Conf}(\mathcal{A}^{n+1},\mathcal{B})$ to geometric crystals in \cite[Appendix B]{GS1} and \cite[Section 7.1]{GS2}. 
\end{Rmk}

Having in mind various deformations of tensor multiplicities (see \cite[Section 7.3]{BKII}) on the one hand, the unipotent bicrystal realization of $H^{\star n}$ on the other hand, we construct a large (albeit finite) set ${\bf C}_n={\bf C}_n(G)$ of positive functions  on $H^{\star n}$ (see Section \ref{hcc}), and sometimes refer to elements of ${\bf C}_n$ as {\em central charges}. 
For $n=2$, ${\bf C}_2=\{c_0,c_1,c_2\}$ so that $\Phi_{H\star H}=c_0+c_1+c_2$, which, viewed as a $U^3$-invariant rational function on $G\times G$ is given by 
$$\Phi_{BK}(g_1)+\Phi_{BK}(g_2)-\Phi_{BK}(g_1g_2), g_1,g_2\in G,$$
where $\Phi_{BK}$ is a $U\times U$-equivariant rational function on $G$ introduced by David Kazhdan and the first author in \cite{BKII}. The set ${\bf C}_3$ consists of $8$ functions, the potential $\Phi_{H\star H\star H}$ is the sum of $4$ of them (see Example \ref{ex:n=3}, also we compute ${\bf C}_n$ explicitly for $G=\GL_2$ in Example \ref{hccGL2}). 

Since ${\bf C}_n$  emerged naturally as a set of $U^{n+1}$-invariant functions on $G^n$, we expect the central charges to behave like integrals for some integral system. 

As a combinatorial application of positivity of central charges, we can deform the tensor product multiplicity $c_{\lambda_1^\vee,\ldots, \lambda_n^\vee}^{\mu^\vee}$ as follows:
\[
c_{\lambda_1^\vee,\ldots, \lambda_n^\vee}^{\mu^\vee;c}(q):=\sum_{a\in \left( H^{\star n} \right)_{\mu^\vee,\lambda_1^\vee,\ldots, \lambda_n^\vee}^t} q^{c^t(a)},
\]
for any central charge $c$ in ${\bf C}_n$ (or even in the positive semifield generated by ${\bf C}_n$, where $\left( H^{\star n} \right)_{\mu^\vee,\lambda_1^\vee,\ldots, \lambda_n^\vee}^t$ denotes the tropical fiber of $H^{\star n}$ and $c^t\colon \left(H^{\star n} \right)\to \mathbb{Z}$ the tropicalization of $c$.

We hope to relate these new polynomials for various choices of $c$ with the LLT-polynomials (see {\em e.g.}, \cite{LLT} ) and Kostka-Foulkes polynomials (see {\em e.g.}, \cite{KF}) which also deform relevant multiplicities.


Now we focus on our second goal, namely new combinatorial-geometric formulas for tensor product and 
and reduction (to Levi subgroups) multiplicities. 
To achieve this goal, we introduce the notion of {\em decorated} reduced word of any element $w\in W$, the Weyl group of $G$.
In particular, to each decorated reduced word $\mathbb{I}$ of the longest element $w_0$, we assign an open embedding  $x_{\mathbb{I}}$ from $(\mathbb{C}^\times)^{m}$ to $U$, where $m=\dim(U)$. To get new combinatorial-geometric formulas for $c_{\lambda^\vee,\nu^\vee}^{\mu^\vee}$, we introduce the notion of {\em decorated $\mathbb{I}$-trails}, which generalize the notion of $\mathbf{i}$-trails in \cite{BZ01}. 

Our main result in this direction is Theorem \ref{comb}, which asserts that  $c_{\lambda^\vee,\nu^\vee}^{\mu^\vee}$ is the number of lattice points in a polytope $P_{\lambda^\vee,\mu^\vee,\nu^\vee}^{\mathbb{I}}$ for any given decorated reduced word $\mathbb{I}$ for $w_0$, whose facets are computed in terms of generalized $\mathbb{I}$-trails.
Taking trivially decorated reduced words, this will recover main results of \cite{BZ01}
In Section \ref{EDC}, we estimate the number of nonequivalent  decorated words for $w_0$, thus estimating  the number of ``significantly different'' formulas for $c_{\lambda^\vee,\nu^\vee}^{\mu^\vee}$. 

Since $U$ is a cluster variety (see {\em e.g.}, \cite{BFZ}), we can treat our new open embeddings $x_{\mathbb{I}}$'s as certain toric coordinates on $U$, thus generalizing the chamber ansatz (\cite{BFZ96,BZ97, FZ}). Using a general (probably well-knwon) result about open embeddings of a split algebraic torus into the affine space (see Proposition \ref{inversemapapp}), we obtain a set ${\bf X}_{\mathbb{I}}$ of $m=\dim(U)$ irreducible polynomials on $U$, which suggests the following
\begin{Con}
    For any decorated reduced word for $w_0$, the set ${\bf X}_{\mathbb{I}}$ is a cluster for $U$ mutation-equivalent to an initial cluster attached to any reduced word ${\bf i}$ for $w_0$.
\end{Con}

We generalize this to any $w\in W$ in Section \ref{Factorpro}.

All those multiplicities we just computed in many ways are particular case of the following algebro-geometric problem: given a $G^\vee$-variety $X^\vee$, compute the multiplicities in the coordinate algebra $\mathbb{C}[X^\vee]$ viewed as a $G^\vee$-module. Indeed, if $X^\vee=(G^\vee/U^\vee)^n$, then
\[
  \mathbb{C}[X^\vee]=\mathbb{C}[G^\vee/U^\vee]^{\otimes n}\cong \bigoplus_{\lambda_1^\vee,\cdots, \lambda_n^\vee\in P^\vee_+} V_{\lambda^\vee_1}\otimes\cdots \otimes V_{\lambda^\vee_n},
\]
which bring us back the solved tensor product multiplicities problem. Furthermore, if $X^\vee=G^\vee/U^\vee$ and $\mathbb{C}[X]$ viewed as an $H^\vee$-module, then
\[
  \mathbb{C}[X^\vee]=\bigoplus_{\beta^\vee\in P^\vee_+}  \mathbb{C}[X^\vee]_{\beta^\vee}
\]
where $\mathbb{C}[X^\vee]_{\beta^\vee}$ is the $\beta^\vee$-weight space of $\mathbb{C}[X^\vee]$ and 
\[
  \mathbb{C}[X^\vee]_{\beta^\vee}=\bigoplus_{\lambda^\vee\in P^\vee_+}V_{\lambda^\vee}(\beta^\vee).
\]
More generally, if $\widetilde{G}^\vee$ is any complex reductive group containing $G^\vee$ as a subgroup, and $\widetilde{X}^\vee=\widetilde{G}^\vee/\widetilde{U}^\vee$ and $\mathbb{C}[\widetilde{X}^\vee]$ viewed as a $G^\vee$-module, then
\[
  \mathbb{C}[\widetilde{X}^\vee]=\bigoplus_{\lambda^\vee\in P^\vee_+}  \mathbb{C}[\widetilde{X}^\vee]_{\lambda^\vee}\otimes V_{\lambda^\vee}
\]
where $\mathbb{C}[\widetilde{X}^\vee]_{\lambda^\vee}=\bigoplus_{\widetilde{\lambda}^\vee\in \widetilde{P}}\Hom_{G^\vee}(V_{\lambda^\vee},V_{\widetilde{\lambda}^\vee})$. To address these and other multiplicity problems we suggest following

\begin{Con}\label{conmirror}
  For any complex affine $G^\vee$ variety $X^\vee$, there is a momoid $A_{X^\vee}$ in $\mul_G$   such that
  \[
    \mathbb{C}[X^\vee]= \mathcal{V}(A_{X^\vee}).
  \]
  Moreover, the assignment $X\to A_{X^\vee}$ is functorial.
\end{Con}

In fact, this functor should be the ``inverse'' of our functor ${\mathcal V}$ restricted to monoids $A$ in $\mul_G$. The conjecture essentially asserts that the functor $A\mapsto {\mathcal V}(A)$ which sends monoids in $\mul_G$ to $G^\vee$-algebras is ``surjective'', {\em i.e.}, the coordinate algebras of all $G^\vee$-modules are in its range. We will provide more evidence for this conjecture in the near future. 

We conclude with the possible relation of this conjecture with mirror symmetry based on \cite{LT}. Namely, we claim that for any monoid $A\in \mul_G$, the variety $Z_A:=A\times_H G/U$ has a natural Landau-Ginzburg potential $\Phi_{Z_A}$. We expect that in the notation of conjecture \ref{conmirror} the variety $Z_A$ with potential $\Phi_{Z_A}$ is a mirror of $X^\vee$ for $A=A_{X^\vee}$. In particular, since $H\star H$ fibered over $H$ (without additional fibration), $H$ is an algebra in $\mul_G$. Taking $X^\vee=G^\vee/U^\vee$, so that $A_{X^\vee}=H$, we see that $Z_H=G/U$ is the mirror of $G^\vee/U^\vee$ according to \cite{LT}. This would give yet an other relation between the Langlands duality and mirror symmetry.

The paper is organized as follows: Section \ref{posvar&afftop} provides background on positivity theory and tropical affine varieties. In Section \ref{Doubleasposvar}, we introduce new toric charts on double Bruhat cells coming from double Bruhat cell embeddings. Section \ref{Geometric Multiplicities} presents the category of geometric multiplicities and our main results on geometric lift. Section \ref{tensor mul} shows how to pass from geometric multiplicities to tensor multiplicities. Section \ref{combforten} presents the combinatorial expressions using the new toric chart we get in Section \ref{Doubleasposvar}. Section \ref{reduction multi} is about the geometric lift of reduction multiplicities. The remaining sections contain the proofs of all the results in this paper.

{\bf Acknowledgements.} We are grateful to A. Alekseev, B. Hoffman and J. Lane for their useful comments and discussions. 
We also want express our gratitude for the hospitality and support of the Simons Center for Geometry and Physics, where this project has started during our visit in 2018. 

\section{Affine Tropical Varieties and Positive Varieties}\label{posvar&afftop}

\subsection{Affine tropical varieties}

In this section, we will introduce the notion of {\em affine tropical variety}, which is an analog of affine variety in the `tropical word'. For any subring $R$ of $\mathbb{R}$, denote by $R_+:=R\cap \mathbb{R}_+$ the semi-subring. Given subsets $C$ and $D$ of free $R$-modus $V$ and $V'$ respectively, a map $\phi\colon C\to D$ is {\em piecewise $R$-linear} if there is a piecewise $R$-linear $R$-module homomorphism $\widetilde{\phi}\colon V\to V'$ such that $\widetilde{\phi}\big|_C=\phi$.

\begin{Def}\label{afftrop}
  Fix a subring $R$ of $\mathbb{R}$. A $m$-dimensional {\em affine tropical variety} $\mathcal{C}$ over $R$ is a family of pairs $(C_{\sbt}, j_{\sbt})$ with index set $\Theta$ and a $m$-dimensional free $\mathbb{R}$-module $V$,  where for any $\theta\in \Theta$, $C_{\theta}$ is a set with an $R_+$-action and $j_{\theta}\colon C_{\theta} \to V$ is an injective map, which is called a \emph{tropical chart}, such that
  \begin{itemize}
    \item[(i)] The map $j_{\theta}$ commutes with the $R_+$-action;

    \item[(ii)] The transition map $j_{\theta}^{-1}\circ j_{\theta'}\colon C_\theta\to C_{\theta'}$ is a piecewise $R$-linear bijection.
   \end{itemize}
  We say that an affine tropical variety $\mathcal{C}$ is {\em convex} if there exist a $\theta\in \Theta$ such that the subset $j_{\theta}(C_{\theta})$ of $V$ is a $R_+$-submodule of the free module $V$. The {\em coordinate algebra} $R[\mathcal{C}]$ of $\mathcal{C}$ is the pull-back of piecewise $R$-linear functions of $V$ along the injective map $j_\theta$. Since $j_{\theta}^{-1}\circ j_{\theta'}$ is piecewise $R$-linear, the coordinate algebra $R[\mathcal{C}]$ is independent of charts.
\end{Def}

Similarly to algebraic varieties, for an affine tropical variety ${\mathcal C}$ over $R$, tensoring with $R'\supset R$, one obtains its $R'$-points
\[
   {\mathcal C}(R'):=\{(C_\theta\underset{R_+}{\otimes} R'_+, j_{\theta})\mid \theta\in \Theta\}.
\]
A morphism ${\bm f}$ of affine tropical varieties $\mathcal{C}=(C_{\sbt}, j_{\sbt})$ and $\mathcal{D}=(D_{\sbt}, k_{\sbt})$ over $R$ is a family of piecewise $R$-equivariant maps $f_{\sbt}$ with index set $(\Theta_{\mathcal{C}}, \Theta_{\mathcal{D}})$, such that the following diagram is commutative.
\begin{equation}\label{comm}
  \begin{tikzcd}
    C_\theta \arrow[r, " j_{\theta}^{-1}\circ j_{\theta'}"] \arrow[d, "f_{\theta, \vartheta}"'] &[1.3em] C_{\theta'} \arrow[d, "f_{\theta', \vartheta'}"] \\[1em]
    D_\vartheta \arrow[r, "k_{\vartheta}^{-1}\circ k_{\vartheta'}"] & D_{\vartheta'}
  \end{tikzcd}
\end{equation}
Affine tropical varieties over $R$ form a category $\mathbf{AffTropVar}(R)$.

\begin{Def}\label{finite}
  Let ${\bm f}$ be a morphism of affine tropical varieties $\mathcal{C}=(C_{\sbt}, j_{\sbt})$ and $\mathcal{D}=(D_{\sbt}, k_{\sbt})$ over $\mathbb{Z}$. For $\xi\in D_{\vartheta}$, define the {\em multiplicity} of $\xi$ over ${\bm f}$ as 
  \[
    \dim \mathbb{C}[f^{-1}_{\theta,\vartheta}(\xi)],
  \]
  where $\mathbb{C}[X]$ is the linearization of the set $X$. Note that the multiplicity of $\xi$ over ${\bm f}$ doesn't depend on the charts since by \eqref{comm}, we have:
  \[
    \dim \mathbb{C}[f^{-1}_{\theta,\vartheta}(\xi)]=\dim \mathbb{C}[f^{-1}_{\theta',\vartheta'}(k_{\vartheta}^{-1}\circ k_{\vartheta'}(\xi))].
  \]
  The morphism ${\bm f}$ is {\em finite} if every $\xi\in D_{\vartheta}$ has finite multiplicity.
\end{Def}

\begin{Ex}
  Define an affine tropical variety over $R$ as $\mathcal{I}_R:=\{(R_+, j\colon R_+\hookrightarrow R)\}$, which we refer to as the {\em trivial} affine tropical variety over $R$. Given affine tropical varieties $\mathcal{C}=(C_{\sbt}, j_{\sbt})$ and $\mathcal{D}=(D_{\sbt}, k_{\sbt})$ over $R$, the product
  \[
    \mathcal{C}\times \mathcal{D}:=\big\{(C_{\theta}\times D_{\vartheta}, j_{\theta}\times k_{\vartheta})\ \big|\ (\theta,\vartheta)\in \Theta_\mathcal{C}\times \Theta_\mathcal{D}\big\}
  \]
  is an affine tropical variety over $R$. For a morphism  ${\bm f}$ of $\mathcal{C}$ and $\mathcal{D}$, one can show that both ${\bm f}(\mathcal{C})$ and ${\bm f}^{-1}(\mathcal{D})$ are affine tropical varieties.
\end{Ex}

\subsection{Positive varieties and tropicalization}

In this section, we first briefly recall basic definitions in positivity theory and then realize tropicalization as a functor from the category of positive varieties with potential to the category of affine tropical varieties.

Consider a split algebraic torus $S\cong \mathbb{G}_{\bf{m}}^n$. Denote the character lattice of $S$ by $S_t=\Hom(S, \mathbb{G}_{\bf{m}})$ and the cocharacter lattice by $S^t = \Hom(\mathbb{G}_{\bf{m}}, S)$. The lattices $S_t$ and $S^t$ are naturally in duality and denote by $\langle \cdot, \cdot \rangle\colon S_t  \times S^t \to \mathbb{Z}$ this canonical pairing. The coordinate algebra $\mathbb{Q}[S]$ is the group algebra (over $\mathbb{Q}$) of the lattice $S_t$, that is, each $f \in \mathbb{Q}[S]$ can be written as
\begin{equation}\label{chi}
  f= \sum_{\chi \in S_t} c_\chi \chi,
\end{equation}
where only a finite number of coefficients $c_\chi$ are non-zero. Following \cite{BKII}, to each positive rational map $\phi\colon S\to S'$, we associate a \emph{tropicalized map} $\Phi^t\colon S^t\to (S')^t$ in the following way:

{\em Case 1.} If $\phi$ is positive regular on $S$, \emph{i.e.}, $\phi$ can be written as \eqref{chi} with all $c_\chi \geqslant 0$, define $\phi^t$ by
\[
  \phi^t\colon S^t \to \mathbb{G}_{\bf{m}}^t = \mathbb{Z}\ :\ \xi\mapsto \min_{\chi; \, c_\chi >0} \langle \chi, \xi \rangle.
\]

{\em Case 2.} If $\phi$ is positive rational on $S$, \emph{i.e.}, $\phi=f/g$ with $f,g$ positive regular functions, then 
\[
  \phi^t:=f^t - g^t.
\]

{\em Case 3.} For $\phi\colon S \to S'$ a positive rational map, define $\phi^t\colon S^t \to (S')^t$ as the unique map such that for every character $\chi \in S'_t$ and for every cocharacter $\xi \in S^t$ we have
\[
  \langle \chi, \phi^t(\xi) \rangle = (\chi \circ \phi)^t(\xi).
\]
A more concrete description is as follows. Let $\phi_1,\dots,\phi_m$ be the components of $\phi$ given by the splitting $S'\cong \mathbb{G}_{\bf{m}}^m$. Then, in the induced coordinates on $(S')^t,$ we have
\[
  \phi^t=(\phi_1^t,\dots,\phi_m^t).
\]

\begin{Def}
  Let $(X,\Phi)$ be an irreducible variety over $\mathbb{Q}$ with a rational function $\Phi$ on $X$. A \emph{rational chart} of $X$ is a birational isomorphism $\theta\colon S \to X$ from a split algebraic torus $S$ to $X$. A rational chart $\theta$ is {\em toric} if $\theta$ is an open embedding. A chart $\theta\colon S\to X$ is {\em positive} with respect to $\Phi$ if $\Phi\circ \theta$ is a positive rational function on $S$. Two charts $\theta_1\colon S_1\to X$ and $\theta_2\colon S_2\to X$ are called \emph{positively equivalent} if $\theta_1^{-1}\circ\theta_2\colon S_2 \to S_1$ and $\theta_2^{-1}\circ\theta_1\colon S_1\to S_2$ are positive rational maps. A \emph{positive variety with potential} is a triple $(X,\Phi, \Theta_X)$, where $\Theta_X$ is a set of positive equivalent charts who are positive with respect to $\Phi$. Given a positive chart $\theta\colon S\to X$ of $(X,\Phi, \Theta_X)$, denote by
  \[
    (X,\Phi, \theta)^t :=\left\{\xi\in \Hom(\mathbb{G}_{\bf{m}},S) \mid \Phi^t(\xi)\geqslant 0\right\},
  \]
  the tropicalization of $(X,\Phi, \theta)$. For convenience, we define $0^t:=+\infty$. If $\theta$ is toric and $\Phi$ is regular, the set $(X,\Phi,\theta)^t$ is a convex cone in $\Hom(\mathbb{G}_{\bf{m}},S)$.
\end{Def}

A morphism $f\colon (X,\Phi,\Theta_X)\to (Y,\Phi',\Theta_Y)$ of positive varieties with potential is a rational map $f\colon X\to Y$ such that the rational function $\Phi-f^*\Phi'$ is positive, and for some (equivalently any) $\theta_X\in \Theta_X$ and $\theta_Y\in \Theta_Y$, the rational map $\theta_Y^{-1} \circ f\circ \theta_X\colon S\to S'$ is positive.
 
Denote by $\mathbf{PosVarPot}(\mathbb{Q})$ the category of the positive varieties with potential over $\mathbb{Q}$.

\begin{Pro}\label{prop;equivofcats}
  Let $(X,\Phi, \Theta)$ be a positive variety with potential. Fix a splitting of $S$. Define $X^t_{\Phi}$ as
  \[
    X^t_{\Phi}:=\Big\{X_\theta:=(X,\theta,\Phi)^t,j_{\theta}\colon X_\theta\hookrightarrow \Hom(\mathbb{G}_{\bf{m}},S)\xrightarrow{\sim} \mathbb{Z}^m \ \big| \ \theta \in \Theta  \Big\}.
  \]
  Then $X^t_{\Phi}$ is an affine tropical variety over $\mathbb{Z}$. If $X$ has a positive toric chart with $\Phi$ regular, the affine tropical variety $X^t_{\Phi}$ is convex. In summary, tropicalization defines a functor from $\mathbf{PosVarPot}(\mathbb{Q})$ to $\mathbf{AffTropVar}(\mathbb{Z})$.
\end{Pro}

Note that $(\mathbb{Q}, \id, \theta\colon \mathbb{Q}_{\bf m}\to \mathbb{Q})$ is a positive variety with potential and $\mathbb{Q}_{\id}^{t}=\mathcal{I}_{\mathbb{Z}}$. Thus potential $\Phi_X$ on $(X, \Theta)$ can be viewed as a morphism of positive varieties with potential in the following way:
\[
  \Phi_X\colon (X,\Phi_X,\Theta) \to (\mathbb{Q}, \id, , j\colon \mathbb{Q}_{\bf m}\to \mathbb{Q}).
\]

Let $f$ be a morphism of two positive varieties with potential $(X,\Phi_X, \Theta_X)$ and $(Y,\Phi_Y,\Theta_Y)$. Denote by
\[
  f^{-t}_{\theta_X,\theta_Y}(\xi):=(f^t_{\theta_X,\theta_Y})^{-1}(\xi)\subset (X,\Phi_X,\theta_X)^t
\]
the pre-image of $\xi\in (Y,\Phi_Y,\theta_Y)^t$ of the tropical function $f^t_{\theta_X,\theta_Y}\colon (X,\Phi_X,\theta_X)^t\to (Y,\Phi_Y,\theta_Y)^t$. We sometimes write $f^{-t}(\xi)$ instead if the positive chart we choose is clear from the context. Note that $f^{-t}(\xi)$ is not an affine tropical variety in general.

\section{Double Bruhat Cells as Positive Varieties}\label{Doubleasposvar}

In this section, we will introduce the so called  Berenstein-Kazhdan potential $\Phi_{BK}$ on a Borel subgroup $B^-$ of a semi-simple algebraic group $G$, and then introduce a positive structure for $(B^-, \Phi_{BK})$. 

\subsection{Double Bruhat cell embeddings}
First of all, let us fix the notation. For a semi-simple algebraic group $G$, choose a Cartan subgroup $H$, and a pair of opposite Borel subgroups $B$ and $B^-$ such that $B\cap B^-= H$. Denote by $U$ and $U^-$ respectively the unipotent radicals of $B$ and $B^-$. Let $r=\rk(\Lie{G})$ and $I$ be the set of vertices of Dykin diagram of $G$. Denote by $e_i,\alpha_i^\vee, f_{i}$ the  Chevalley generators of $\mathfrak{g}$ for $i\in I$, and by $\alpha_i$ the corresponding simple root of $\alpha_i^\vee$. For each $i\in I$, let $\varphi_i \colon \SL_2\to G$ be the canonical homomorphism corresponding to the simple root $\alpha_i$; set
\[
  x_i(t)=\varphi_i\begin{bmatrix}
    1 & t \\
    0 & 1
  \end{bmatrix}\in U,\quad
  y_{i}(t)=\varphi_i\begin{bmatrix}
    1 & 0 \\
    t & 1
  \end{bmatrix}\in U^{-}, \quad
  t^{\alpha_i^\vee}=\varphi_i\begin{bmatrix}
    t & 0\\
    0 & t^{-1}
  \end{bmatrix}\in H.
\]
Moreover, define
\[
  x_{-i}(t)=y_i(t)t^{-\alpha_i^\vee}=\varphi_i\begin{bmatrix}
    t^{-1} & 0 \\
    1 & t
  \end{bmatrix}\in B^-.
\]
For each $i\in I$, define the elementary (additive) character  $\chi_i$ of $U$ by
\[
  \chi_i(x_j(t))=\delta_{ij}\cdot t,\quad \text{~for~} t\in \mathbb{G}_{\bf{m}}. 
\]
Denote by $\chi^{\st}=\sum \chi_i$ the standard character of $U$. Define an anti-automorphism $(\cdot)^T\colon G\to G$ by:
\[
  x_i^T(t)=y_{i}(t),\quad h^T=h, \quad \text{~for~} t\in \mathbb{G}_{\bf{a}} \text{~and~} h\in H;
\]
and define  the `positive inverse', which is an anti-automorphism $(\cdot)^\iota\colon G\to G$ by:
 \begin{equation}\label{eq:iota}
     (x_i(t))^\iota=x_i(t),\quad (y_i(t))^\iota=y_i(t), \quad h^\iota=h^{-1}, \text{~for~} t\in \mathbb{G}_{\bf{a}} \text{~and~} h\in H.
 \end{equation}
Let $W=N_G(H)/H$ be the Weyl group of $G$, and $\overline{s_i}=x_i(-1)y_i(1)x_i(-1)$ be a lift of the reflection $s_i$ generated by simple root $\alpha_i$. Note that $\overline{s_i}$'s satisfy the braid relation. Therefore we associated to each $w\in W$ a lift $\overline{w}:=\overline{s_{i_1}}\cdots \overline{s_{i_l}}$ where $w=s_{i_1}\cdots s_{i_l}$ is a reduced expression of $w$. Let $w_0$ be the longest element in $W$ with $\ell(w_0)=m$. Let $G_0 = U^-HU$ be the set of elements that admit Gaussian decomposition; write $x=[x]_-[x]_0[x]_+$ for $x\in G_0$ with respect to the Gaussian decomposition.

Denote by $L^{v,w}=U\overline{v}U\cap B^{-}wB^{-}$ a {\em reduced} double Bruhat cell for $(v,w)\in W\times W$. 
\begin{Thm}\label{open}
  Let $v,w\in W$ such that $\ell(v)+\ell(w)=\ell(v^{-1}w)$, the following maps
  \[
  	\xi^{v,w}\colon L^{v,w}\hookrightarrow L^{e,v^{-1}w}\ :\ x\to [\overline{v}^{-1}x]_+; \quad \xi_{v,w}\colon L^{w^{-1},v^{-1}}\hookrightarrow L^{w^{-1}v,e}\ :\ x\to [x\overline{v}]_-[x\overline{v}]_0
  \]
  are open embeddings.
\end{Thm}
\begin{proof}
  We will only show the statement for $\xi^{v,w}$. First of all, we show that the map $\xi^{v,w}$ is well-defined. Since $x\in U\overline{v}U$, we have
  \begin{equation}\label{unique}
    \overline{v}^{-1}x\in \overline{v}^{-1}U\overline{v}U\subset B^-U,  
  \end{equation}
  which shows that $[\overline{v}^{-1}x]_+$ is well-defined. Since $x\in B^-wB^-$, 
  \[
    \overline{v}^{-1}x \in v^{-1}B^-wB^- \subset B^-v^{-1}B^-wB^-=B^-v^{-1}w B^{-}=B^-v^{-1}wB^-. 
  \]
  Here we use that $\ell(v)+\ell(w)=\ell(v^{-1}w)$. Thus $[\overline{v}^{-1}x]_+\in B^-v^{-1}wB^-$. The unique factorization \eqref{unique} implies injectivity of $\xi$. What remains is to show that  $\xi$ maps an open subset $L^{v,e}\cdot L^{e,w}\subset L^{v,w}$ onto an open subset $L^{e,v^{-1}}\cdot L^{e,w}$ of $L^{e,v^{-1}w}$. Denote by $x=x_-\cdot x_+\in L^{v,e}\cdot L^{e,w}$, then we have
  \[
  	\xi(x)=[\overline{v}^{-1}x_-x_+]_+=[\overline{v}^{-1}x_-]_+x_+=\iota\circ\eta(x_-)x_+,
  \]
  where $\eta^v\colon L^{v,e}\to L^{e,v^{-1}}$ sending $x$ to $[\overline{v}^{-1}x]_+^\iota$ is the biregular isomorphism defined in \cite[Defintion 4.6]{BZ01}. Thus we conclude that $\xi^{v,w}$ is an open embedding.
\end{proof}

\subsection{Decorated reduced words}
A {\em word} of length $n$ is an element in $I^n$, where $I$ is an index set. A {\em reduced} word for $w$ in the Weyl group $W$ is a word $\mathbf{i}=(i_1,\ldots, i_n)$ such that if $w=w_{\mathbf{i}}$ and $\ell(w)=n$, where $w_{\mathbf{i}}=s_{i_1}\cdots s_{i_n}$. Denote by $R(w)$ the set of reduced word for $w\in W$. A {\em double} word is a word indexed by $-I\cup I$. For any double word $\mathbf{i}=(i_1,\ldots,i_n)$, denote by $-\mathbf{i}:=(-i_1,\ldots,-i_n)$ and $\mathbf{i}^{op}:=(i_n,\ldots,i_1)$.

Denote by $\mathbf{i}_+$ (resp. $\mathbf{i}_-$) the subword of a double word $\mathbf{i}$ contains of all positive (resp. negative) entries of $\mathbf{i}$. A {\em double reduced} word for $(v,w)$ is a double word $\mathbf{i}=(i_1,\ldots,i_n)$ such that $-\mathbf{i}_-$ (resp. $\mathbf{i}_+$) is a reduced word for $v$ (resp. $w$). Denote by $R(v,w)$ the set of double reduced words for $(v,w)\in W\times W$.

Given a pair of subsets $(K,L)$ of $[1,n]$ satifying $|K|=|L|$, we have the following bijections $\sigma_{K,L}$ and $\rho_{K,L}$ from $K$ to $L$:
\begin{align*}
	\widetilde{\sigma}_{K,L}&\colon K\to  L \ :\ k\mapsto l, \text{~where~} l \text{~is uniquely determined by~} |[1,l]\cap L|=|[1,k]\cap K|,\\
	\widetilde{\rho}_{K,L}&\colon K\to  L \ :\ k\mapsto l, \text{~where~} l \text{~is uniquely determined by~} |L|-|[1,l]\cap L|+1=|[1,k]\cap K|,
\end{align*}
which combine into the following permutation of $[1,n]$: 
\[
  \sigma_{K,L}(l)=
  \begin{cases}
    \widetilde{\rho}_{K,L}(k), &\text{if~} k\in K;\\
    \widetilde{\sigma}_{\overline{K},\overline{L}}(k), &\text{if~} k\in \overline{K}.
  \end{cases}.
\]
Note for $K=L=[1,n]$, the permutation $\widetilde{\sigma}_{K,L}=\id$ and $\widetilde{\rho}_{K,L}(k)=n+1-k$, which is the permutation that revers the order. Also note that $\widetilde{\sigma}_{K,L}\circ \widetilde{\sigma}_{L,K}=\id_K$ and $\widetilde{\rho}_{K,L}\circ \widetilde{\rho}_{L,K}=\id_K$.

Given a reduced word $\mathbf{i}$ and a subset $K\subset [1,\ell(\mathbf{i})]$, denote by $\mathbf{i}_K$ the subword of $\mathbf{i}$ indexed by $K$. Denote by $\overline{K}:=[1,\ell(\mathbf{i})]\backslash K$ the complement of $K$. The subset $K\subset [1,\ell(\mathbf{i})]$ is {\em compatible} with $\mathbf{i}$ if
\begin{equation}\label{decor}
  w_{\mathbf{i}_K}w_{\mathbf{i}_{\overline{K}}}=w_{\mathbf{i}}.
\end{equation}
For any $k\leqslant \ell(\mathbf{i})$, the subset $[1,k]\subset [1,\ell(\mathbf{i})]$ is compatible with $\mathbf{i}$. Note that if $(\mathbf{i},K)$ is compatible, then $(\mathbf{i}^{op},\overline{K})$ is compatible. It is an interesting problem to find all compatible $K$ for a given $\mathbf{i}$.

A {\em decorated} word is a triple $(\mathbf{i}, K,L)$, where $\mathbf{i}$ is a word, and $K,L$ are subsets of $[1,\ell(\mathbf{i})]$ such that $|K|=|L|$. A {\em decorated} reduced word for $w\in W$ is a decorated word $(\mathbf{i},K,L)$ such that $\mathbf{i}\in R(w)$, and $(\mathbf{i},K)$ is compatible. Denote by $\widehat{R}(w)$ the set of decorated reduced word for $w\in W$. To each decorated word, define a map by the formula
\[
    \mathbb{I}(\mathbf{i},K,L)=(-\mathbf{i}_K^{op}) \shuffle_L  \mathbf{i}_{\overline{K}},
\]
where $\mathbf{i} \shuffle_L \mathbf{j}$ with $|\mathbf{i}|=|L|$ is a shuffle of $\mathbf{i}$ and $\mathbf{j}$ by putting $\mathbf{i}$ at the position $L$ and $\mathbf{j}$ at the rest in order. The map $\mathbb{I}$ can be described by using $\sigma_{K,L}$ and $\rho(K,L)$:
\[
  \mathbb{I}(i_k)=
  \begin{cases}
    -i_{\sigma_{K,L}(k)}, & \text{if~} k\in K;\\
    i_{\sigma_{K,L}(k)}, & \text{if~} k\in \overline{K}.
  \end{cases}
\]
\begin{Lem}
  The image of $\widehat{R}(w)$ under the map $\mathbb{I}$ is 
  \[
    \mathbb{I}(\widehat{R}(w))=\bigsqcup R(v_1,v_2), 
  \]
  where the disjoint union is over pairs $(v_1,v_2)$ satisfying $v_1^{-1}v_2=w$ and $\ell(v_1)+\ell(v_2)=\ell(w)$.
\end{Lem}

For simplicity, in what follows, we will refer to $\mathbb{I}$ as the {\em associated double word} of $(\mathbf{i},K,L)$. Note that in general the map $\mathbb{I}$ is not injective. 

\begin{Ex}\label{Exforsl4}
  a) For $G=\SL_4$, let $\mathbf{i}=(1,2,3,1,2,1)$ and a decoration $(K=[1,3],L=\{1,2,5\})$. By definition, we have $v_k=s_3s_2s_1$ and $w_k=s_1s_2s_1$. And the double reduced word is $(-3,-2,1,2,-1,1)$.

  b) Denote by $c$ a repetition-free element of length $r$, {\em i.e.}, there exists $n$ such that $c^n$ is reduced. A Coxter element is repetition-free. Fix a reduced word $\mathbf{j}:=(i_1,\ldots,i_r)$ for $c$ and extend it to a reduced word $\mathbf{i}$ by repeating $\mathbf{j}$. Denote by $K_k:=[kr+1,(k+1)r]$ and $L:=[1,r]$. Then it is clear that $\mathbb{I}(\mathbf{i},K_k,L)=\mathbb{I}(\mathbf{i},K_l,L)$ for $k,l<n$.

  c) For any word $\mathbf{i}=(i_1,\ldots,i_k,i_{k+1},\ldots,i_n)$ such that $s_{i_k}s_{i_{k+1}}=s_{i_{k+1}}s_{i_k}$, denote by 
  \[
  	\mathbf{i}'=(i_1,\ldots,i_{k+1},i_k,\ldots,i_n).
  \]
  Then it is easy to check $\mathbb{I}(\mathbf{i},[1,k],L)=\mathbb{I}(\mathbf{i},[1,k-1]\cup \{k+1\},L)$ for any $L$ satisfying $|L|=k$.
\end{Ex}

\subsection{Equivalence of decorated words}\label{EDC}
In what follows we want to explore the set $\widehat{R}(w)$ in more details. Given a subset $K$ of $[1,n]$, a pair $\{k,k'\}$ is {\em separated} by $K$ is neither $\{k,k'\}\subset K$ nor $\{k,k'\}\subset \overline{K}$, otherwise we say they are {\em not separated} by $K$. For a triple $(\mathbf{i},K,L)$, where $\mathbf{i}$ is a reduced word and $K,L$ are subset of $[1,\ell(\mathbf{i})]$ satisfying $|K|=|L|$, a {\em simple move} is a transformation of $(\mathbf{i},K,L)$ in the following form:
\begin{itemize}
	\item[($\tau_1$)] if $\{k,k+1\}$ is not separated by $K$, $|\sigma_{K,L}(k)-\sigma_{K,L}(k+1)|=1$ and $s_{i_{k}}s_{i_{k+1}}=s_{i_{k+1}}s_{i_{k}}$, replace $i_k,i_{k+1}$ in $\mathbf{i}$ by $i_{k+1},i_k$;

	\item[($\tau_2$)] if $\{l,l+1\}$ is separated by $L$ and $i_{\sigma_{L,K}(k)}\neq i_{\sigma_{L,K}(k)}$, replace $L$ by $L'=(L-\{l,l+1\})\cup (\{l,l+1\}-L)$;

 	\item[($\tau_3$)] if $\{1\}\subset L$, replace $K$ by $K'=K\backslash \{ \text{the last term of~} K \}$ and $L$ by $L'=L\backslash \{1\}$.
\end{itemize}
Besides, we have the following move: 
\begin{itemize}
 	\item[($\tau_4$)] if $\{l,l+1\}$ is separated by $L$ and $i_{\sigma_{L,K}(k)} = i_{\sigma_{L,K}(k)}$, replace $L$ by $L'=(L-\{l,l+1\})\cup (\{l,l+1\}-L)$.
\end{itemize}
\begin{Lem}\label{App:Lem1}
	For a decorated word $(\mathbf{i},K,L)$, the moves $\tau_i$'s give rise to the following moves of the double word $\mathbb{I}(\mathbf{i},K,L):=(j_1,\ldots, j_n)$:
	\begin{itemize}
		\item[($\tau_1'$)] if $j_k\cdot j_{k+1}>0$ and $s_{j_{k}}s_{j_{k+1}}=s_{j_{k+1}}s_{j_{k}}$, replace $j_k,j_{k+1}$ by $j_{k+1},j_k$;

		\item[($\tau_2'$)] if $j_k\cdot j_{k+1}<0$ and $|j_{k}|\neq |j_{k+1}|$, replace $j_k,j_{k+1}$ by $j_{k+1},j_k$;

 		\item[($\tau_3'$)] if $i_1<0$, replace $i_1$ by $-i_1$;

 		\item[($\tau_4'$)] if $j_k\cdot j_{k+1}<0$ and $|j_{k}|= |j_{k+1}|$, replace $j_k,j_{k+1}$ by $j_{k+1},j_k$.
	\end{itemize}
\end{Lem}
The proof is straight forward and we omit it here. 

Two decorated reduced words $(\mathbf{i},K,L)$ and $(\mathbf{i}',K',L')$ of $w$ are {\em equivalent} if there is a sequence of decorated words in $\widehat{R}(w)$ such that they are related by simple moves. If $(\mathbf{i},K,L)$ and $(\mathbf{i}',K',L')$ are equivalent decorated reduced word of $w_0$, then by Lemma \ref{App:Lem1} and \cite[Proposition 7.2]{BZ01}, the two charts $x_{\mathbb{I}}$ and $x_{\mathbb{I}'}$ are related by a sequence of monomial change of coordinates, which implies that the combinatoric expressions in Theorem \ref{comb} associated to $(\mathbf{i},K,L)$ and $(\mathbf{i}',K',L')$ differ by a linear transformation. One of the advance of using decorated reduced words is to give more combinatoric expressions for the tensor multiplicities. In what follows, we will try to justify that using decorated reduced words, we do get many more combinatoric expressions (up to linear transformations).

Note that the Dynkin diagram can be partitioned into two sets of non-adjacent vertices. Let $a$ and $b$ be the product of the simple reflections corresponding to the  sets respectively. Note that both $a$ and $b$ are products of commuting reflections. Let $h$ be the Coxter number. If $\mathfrak{g}$ is not of type $A_{2n}$, then we know $w_0=(ab)^{h/2}$. If $\mathfrak{g}$ is of type $A_{2n}$, we have $w_0=(ab)^{(h-1)/2}a=b(ab)^{(h-1)/2}$.
\begin{Lem}\label{App:Lemma}
	Let $w:=aba$ be a reduced expression, where $a$ is a product of commuting simple reflections. Then the number of non-equivalent decorated words of $w$ is greater than or equal to $2^{\ell(a)}$.
\end{Lem}
\begin{proof}
	First, let us assume that the expression of $b$ contains no simple reflections that appear in the expression of $a$. Let $I$ be a reduced word for $a$, and $J$ be a reduced word for $b$. Denote by $n=\ell(a)$. Consider the decorated word $\left( (I,J,I), \emptyset, \emptyset \right)$ and $\mathbb{I}$ the associated double reduced of it. By the fact that $a$ is a product of commuting simple reflections and using move ($\tau_1'$) ($\tau_2'$) and ($\tau_3'$), one can get the following double reduced word:
	\begin{equation}\label{flip}
		(J, -i_1,i_1,\ldots,-i_n,i_n).
	\end{equation}
	Now we have $2^{n}$ ways to use the flip the pairs $i,-i$ in the expression \eqref{flip}. It is not hard to check all this new expression are non-equivalent. If the expression of $b$ contains the simple reflections of $a$, then we need to use ($\tau_4'$) to pass $-i_k$ through $b$, thus we will get more non-equivalent classes. Thus we get the lower bound $2^{\ell(a)}$.
\end{proof}

Similar to this Lemma, for the reduced expression of $w_0$, we have
\begin{Pro}\label{App:prop}
	For a reduced expression of form $w:=(ab)^na^{\varepsilon}$, where $a$ and $b$ are  products of commuting simple reflections and $n\in \mathbb{Z}_+$ and $\varepsilon\in\{ 0,1\}$ . Then the number of non-equivalent decorated words of $w$ is greater than or equal to $n^{\ell(a)}+n^{\ell(b)}$.
\end{Pro}

\begin{Rmk}
	The charts given by the $2^{\ell(a)}$ double reduced words arising from \ref{flip} by using ($\tau_4'$) are birational isomorphic, but not biregular isomorphic. Thus in Theorem \ref{comb}, we get at least $2[h/2]^{[r/2]}$ new combinatoric expression (up to linear transformation), where $h$ is the Coxter number and $r$ is the rank, since the longest element $w_0$ has a form (see \cite{Humph}, for example) as in Proposition \ref{App:prop}.
\end{Rmk}

\subsection{Positive structure for double Bruhat cells}

On Bruhat cell $Bw_0B$, we have the following regular function (Berenstein-Kazhdan potential):
\begin{Def}\label{BKp0}
  On the Bruhat cell $G^{w_0}=Bw_0B$, the BK potential $\Phi_{BK}$ is
  \[
    \Phi_{BK}(uh\overline{w_0}u')=\chi^{\st}(u)+\chi^{\st}(u'), \quad \text{for~} uh\overline{w_0}u'\in G^{w_0}.
  \]
\end{Def}
Since $Bw_0B\cap B_-\hookrightarrow B_-$, so the potential restrict to open dense subset of $B_-$. In what follows, we will introduce a positive structure for $(B^-,\Phi_{BK})$. To a decorated word $(\mathbf{i},K,L)$, let
\[
  \mathbb{I}(\mathbf{i},K,L)=:(j_1,\ldots,j_n)
\]
be the associated double word, define
\[
  x_\mathbb{I}(t_1,\ldots, t_n):=x_{j_1}(t_1)\cdots x_{j_n}(t_n).
\]
Note that $x_{\mathbb{I}}\in L^{v_K, w_K}$ by \cite[Proposition 4.5]{BZ01}, where
\[
	v_K:=w_{\mathbf{i}_K}^{-1};\quad w_K:=w_{\mathbf{i}_{\overline{K}}}.
\]

\begin{Pro}\label{positiveonU}
  Given a decorated reduced word $\left(\mathbf{i},K, L\right)\in \widehat{R}(w)$, the following map
  \[
    \xi_{\mathbb{I}}\colon \mathbb{T}^n\to L^{e,w} \  : \ (t_1,\ldots, t_n) \mapsto [\overline{v_K}^{-1}x_\mathbb{I}(t_1,\ldots, t_n)]_+.
  \]
  is an open embedding. Moreover, for any two associated double reduced words  $\mathbb{I}:=\mathbb{I}(\mathbf{i},K,L )$ and $\mathbb{I}':=\mathbb{I}(\mathbf{i}',K',L')$,  $\xi_{\mathbb{I}}$ and $\xi_{\mathbb{I}'}$ are positively equivalent.
\end{Pro}
\begin{proof}
  Note that by \cite[Proposition 4.5]{BZ01}, the following map
  \[
  	\tau_{\mathbb{I}} \colon \mathbb{T}^n\to L^{v_K,w_K} \  : \ (t_1,\ldots, t_n) \mapsto x_{\mathbb{I}}(t_1,\ldots, t_n)
  \]
  is an open embedding. So that $\xi_{\mathbb{I}}=\xi^{v_K,w_K}\circ\tau_{\mathbb{I}}$ is an open embedding by Theorem \ref{open}.

  Next, we show that $\xi_{\mathbb{I}}$ and $\xi_{\mathbb{I}'}$ are positive equivalent. A decorated word $(\mathbf{i},K,L)$ for $\mathbf{i}=(i_1,\ldots,i_l)$ is called {\em separable} if $K=L=[1,k]$ or $K=L=\emptyset$. Using the commutating relation of \cite[Proposition 7.2]{BZ01}, any associated double reduced word $\mathbb{I}$ is positive equivalent to a separable one by commutating all $x_{j_i}$ for $j_i<0$ to the left one by one. For a separable decorated word $(\mathbf{i},K,L)$, the map $\xi_{\mathbb{I}}$ is actually the composition:
  \[
    \mathbb{T}^l\to L^{v_K,e}\times L^{e,w_K}\xrightarrow{\iota\circ\eta^{v_K}\times \id} L^{e,v_K^{-1}}\times L^{e,w_K} \to L^{e,v_K^{-1}w_K}, 
  \]
  where $\eta^v\colon L^{v,e}\to L^{e,v^{-1}}$ sending $x$ to $[\overline{v}^{-1}x]_+^\iota$. By \cite[Theorem 4.7]{BZ01}, we conclude that any decorated word $(\mathbf{i},K,L)$ is positive equivalent to a separable one $(\mathbf{i}',[1,|K|],[1,|L|])$. Thus $\xi_{\mathbb{I}}$ and $\xi_{\mathbb{I}'}$ are positive equivalent.
\end{proof}
\begin{Rmk}
	For reduced double Bruhat cell $L^{w,e}$, one can come up with a similar statement. Since we don't use it in the paper, we omit it here.
\end{Rmk}

\begin{Ex}(Continues of Example \ref{Exforsl4} a).)
We just write down the map $\xi_{\mathbb{I}}$ explicitly:
  \[
    \xi_{\mathbb{I}}\colon (t_1,\ldots,t_6)\mapsto [\overline{s_3s_2s_1}^{-1}x_{-3}(t_1)x_{-2}(t_2)x_{1}(t_3)x_{2}(t_4)x_{-1}(t_5)x_{1}(t_6)]_+,
  \]
  and the matrix on the right hand side is given by:
  \[
    \begin{bmatrix}
      1 & 2t_6 & t_2+t_4 & t_1\\
      0 & 1 & t_2t_3+t_2t_6^{-1}+t_4t_6^{-1} & t_1t_3+t_1t_6^{-1}\\
      0 & 0 & 1 & t_1t_2^{-1}\\
      0 & 0 & 0 & 1
    \end{bmatrix}.
  \]
\end{Ex}

Since the natural embedding $L^{e,w_0}\to U$ is open, the map $\xi_{\mathbf{i},\sigma}$ gives a toric chart on $U$ for $(\mathbf{i},\sigma)\in \widehat{\bm{R}}(w_0)$. Because of the open embedding $H\times U^T\hookrightarrow B^-$, we get a toric chart on $B^-$ as well. We denote these positive structures by $\Theta_H$, $\Theta_U$ and $\Theta_{B^-}$ for $H$, $U$ and $B^-$ respectively.

\begin{Pro}\cite[Lemma 3.36]{BKII}
  The triple $(B^-, \Phi_{BK}, \Theta_{B^-})$ is a positive variety with potential. 
\end{Pro}

\subsection{Factorization problems}\label{Factorpro}

In this section, we will discuss the inverse map of $x_{\mathbb{I}}$. First of all, let us recall the following well-known general statement:
\begin{Pro}\label{inversemapapp}
	Given an open embedding $\varphi: {\mathbb T}_m\to {\mathbb A}^m$, where ${\mathbb T}_m$ is a split $m$-dimensional algebraic torus, there exist a (unique) finite set $\{f_1,\ldots,f_m\}$ of irreducible polynomials in $m$ variables such that each coordinate function $\varphi^{-1}_k$ is an alternating product of these polynomials.
\end{Pro}
Just for refernces, we put a proof in the Appendix. To make use of this proposition, we will find an open embedding $L^{e,w}\to {\mathbb A}^m$. Denote by
\[
  U(w):=U\cap wU^{-1}w;\quad U^-(w):=U^{-}\cap w^{-1}Uw
\]
the {\em Schubert cell}. We have:
\begin{Lem}\cite{FZ,BZ01}
  The following map is an open embedding:
  \[
    \psi_w\colon L^{w,e}\to U^-(w)\ :\ x\mapsto [\overline{w}^{-1}x]_-.
  \]
\end{Lem}
\begin{proof}
	Note that $\overline{w}^{-1}x\in w^{-1}UwU\subset B_-U$, thus $[\overline{w}^{-1}x]_-$ is well-defined and $[\overline{w}^{-1}x]_-\in U^{-}(w)$. The uniqueness of Gauss decomposition guarantees that $\psi_w$ is injective. The map is open follows from the biregular isomorphism $B_-\backslash (B_-wB_-)\cong U^-(w)$ by \cite[Proposition 2.10]{FZ} and the fact that $L^{w,e}$ is open in $B_-\backslash (B_-wB_-)$.
\end{proof}

Thus we get an open embedding $\psi_{w^{-1}}\circ \xi_{e,w}\colon L^{e,w}\to L^{w^{-1},e}\to U^-(w^{-1})$.

\begin{Lem}\cite[Proposition 2.11]{FZ}
  For a reduced word $\mathbf{i}\in R(w)$, the map below is a well-defined biregular isomorphism:
  \[
    \mathbb{T}^n\to U^-(w)\ :\ (t_1,\ldots,t_m)\mapsto \overline{w}^{-1}x_{i_1}(t_1)\overline{s_{i_1}}\cdots  x_{i_m}(t_m)\overline{s_{i_m}}.
  \]
\end{Lem}

Now combine these lemmas and proposition in beginning of this section, we get:

\begin{Thm}\label{invmap}
  Given a decorated double reduced word $(\mathbf{i},K,L)$ and $\mathbb{I}=\mathbb{I}(\mathbf{i},K,L)$, suppose $x\in L^{e,w}$ can be factored as  $x_{\mathbb{I}}(t_1,\ldots,t_n)$, then $t_k$'s are Laurent monomials of a unique set of regular functions $\{f_1,\ldots,f_n\}$, who are regular on $U^-(w^{-1})$ as well.
\end{Thm}
\begin{Rmk}
	Note that for $K=L=[1,\ell(w)]$ or $K=L=\emptyset$, these function are given by twisted minors, see \cite{FZ,BZ01} for more details. We hope we can find explicit formulas in the general cases. 
\end{Rmk}

To related this to the cluster algebra structure on $L^{e,w}$, we have the following conjecture:
\begin{Con}
  The set $\{f_1,\ldots,f_n\}$ in Theorem \ref{invmap} is a cluster for $L^{e,w}$, {\em i.e.}, the fraction field generated by $\{f_1,\ldots,f_n\}$ is the function field of $L^{e,w}$.
\end{Con}

\section{Geometric Multiplicities}\label{Geometric Multiplicities}

In this section, we will introduce the category $\mul_G$ of geometric multiplicities, which is monoidal with a non-trivial associator (without unite). 

\subsection{\texorpdfstring{$U\times U$}{U*U} varieties and unipotent bicrystals}\label{Uuub}

In this section, we briefly recall the basic definitions about $U\times U$ varieties and unipotent bicrystals in \cite{BKII} and introduce the notion of trivializable unipotent bicrystals. 

\begin{Def}\label{uuvariety}
  A $U\times U$-{\em variety} ${\bm X}$ is a pair $(X,\alpha)$, where $X$ is an irreducible affine variety over $\mathbb{Q}$ and $\alpha\colon U\times X\times U\to X$ is a $U\times U$-action on $X$, such that group $U$ acts (both action) freely on $X$. The {\it convolution product} $*$ of $U\times U$-varieties ${\bm X}=(X,\alpha)$ and ${\bm Y}=(Y,\alpha')$ is ${\bm X}*{\bm Y}:=(X * Y,\beta)$, where the variety $X * Y$  is the quotient of $X\times Y$ by the following left action of $U$ on $X \times Y$:
  \[
    u(x,y)=(xu^{-1},u y).
  \]
  And the action $\beta\colon U\times X*Y\times U\to X*Y$ is defined by $u(x*y)u'=(ux)*(yu')$.
\end{Def}
\begin{Ex}
  It is clear that the group $G$ itself is a $U\times U$-variety with left and right multiplication as $U\times U$ action. Denote by
  \[
    G^{(n)}:=G*\cdots *G, \quad \text{for~} n\geqslant 2
  \]
the convolution product of $n$-copies of $G$'s.
\end{Ex}
\begin{Def}
  For a $U\times U$-variety ${\bm X}$ and $\chi\colon U\to \mathbb{A}^1$ a character, a function $\Phi$ on $X$ is $\chi$-{\em linear} if
  \begin{equation}
    \Phi(u\cdot x\cdot u')=\chi(u)+\Phi(x)+\chi(u'), \forall x\in X, u,u'\in  U.
  \end{equation}
  A {\em $(U\times U, \chi)$-bicrystal} is a triple $({\bm X},{\bm p}, \Phi)$, where  ${\bm X}$ is a $U\times U$-variety, and ${\bm p}\colon X\to G$ is a $U\times U$-equivariant morphism, and $\Phi$ is $\chi$- linear function. We refer to the pair $({\bm X},{\bm p})$ as {\em unipotent bicrystal}. The convolution product is defined by
  \[
    ({\bm X},{\bm p}, \Phi_X)*({\bm Y},{\bm p}', \Phi_{Y}):=({\bm X}*{\bm Y},{\bm p}'', \Phi_{X*Y}), 
  \]
  where ${\bm p}''\colon X*Y\to G$ is defined by ${\bm p}''(x*y)={\bm p}(x){\bm p}'(y)$ and $\Phi_{X*Y}(x*y)=\Phi_X(x)+\Phi_Y(y)$.
\end{Def}

\begin{Rmk}
  There are more $U\times U$-invariant functions on $G^{(n)}$ arising from the potential $\Phi_{BK}$, which we call {\em higher central charges}. See more details in Section \ref{hcc}.
\end{Rmk}

\begin{Ex}
	By Definition \ref{BKp0}, the BK potential $\Phi_{BK}$ is a $\chi^{\st}$-linear function on the $U\times U$ variety $G$. Therefore, the $U\times U$ variety $G^{(n)}$ is a $(U\times U,\chi^{\st})$-bicrystal with ${\bm p}\colon G^{(n)}\to G$ by sending $g_1*\cdots *g_n$ to $g_1\cdots g_n$, and the $\chi^{\st}$-linear function, or potential is given by
	\[
	  \Phi_{G^{(n)}}(g_1*\cdots*g_n):=\sum \Phi_{BK}(g_i).
	\]
\end{Ex}

\begin{Def}
	The {\em highest weight map} $\hw$ of $G$ is the following $U\times U$-invariant rational morphism
\begin{equation}\label{eq:pi0}
  \hw\colon Bw_0B\to H \ : \ uh\overline{w_0}u'\mapsto h.
\end{equation}
\end{Def}

To a $(U\times U,\chi^{\st})$-bicrystal $({\bm X}, {\bm p}, \Phi)$, the {\em central charge} of $({\bm X},{\bm p},\Phi)$ is the $U\times U$-invariant function:
\begin{equation}\label{eq:Delta}
  \Delta_X(x):=\Phi(x)-\Phi_{BK}({\bm p}(x)), \forall x\in X.
\end{equation}
Assume that $U\backslash X/U$ is an affine variety in the following. Since both $\Delta_X$ and $\hw_X:=\hw\circ {\bm p}$ are $U\times U$-invariant, they descents to functions $\overline{\Delta}_X$ and $\overline{\hw}_{X}$ on $U\backslash X/U$ respectively. Now consider the affine variety:
\[
	\widetilde{X}:=U\backslash X/U\times_H G,
\]
where the fiber product is over $\overline{\hw}_X$ and $\hw$. The variety $\widetilde{X}$ gets an $U\times U$ action on $G$:
\[
  u\cdot (\overline{x},g) \cdot u'\mapsto (\overline{x}, ugu').
\] 
Define a $\chi^{\st}$-linear function $\widetilde{\Phi}$ on $\widetilde{X}$ by
\[
  \widetilde{\Phi}(\overline{x},g):=\overline{\Delta}_X(\overline{x})+ \Phi_{BK}(g), \quad \text{~for~} (\overline{x}, g)\in U\backslash X/U\times_H G, 
\]
Denote by $p_2$ is the projection $\widetilde{X}$ to the second factor $G$. All these make the triple $(U\backslash X/U\times_H G,\widetilde{\Phi}, p_2)$ into a  unipotent bicrystal.
\begin{Def}
	A $(U\times U,\chi^{\st})$-bicrystal $({\bm X},{\bm p},\Phi)$ is {\em trivializable} if the following map is a birational isomorphism of $(U\times U,\chi^{\st})$-bicrystals
	\begin{equation}\label{trivializable}
	  \varphi\colon X\xrightarrow{\sim} U\backslash X/U\times_H G \ : \ x\mapsto (\overline{x},{\bm p}(x)).
	\end{equation}
	Denote by $\mathbf{TriUB}_G$ the category of trivializable $(U\times U,\chi^{\st})$-bicrystals over $G$.
\end{Def}
The following proposition shows that $U\backslash X/U$ can be realized a subvariety of $X$:
\begin{Pro}\label{Thm5}
	For a trivializable $(U\times U,\chi^{\st})$-bicrystal $({\bm X},{\bm p},\Phi)$, the following natural map is a birational isomorphism of varieties
  	\[
    	{\bm p}^{-1}\left(\phi(H)\right)\to  U\backslash X/U.
  	\]
  	where $\phi\colon H\to G$ is the natural rational lift of $\hw\colon G\to H$ given by $ \phi(h)=h\overline{w_0}\in Bw_0B\subset G$. Moreover, we have $X\cong Y\times_H G\cong U\backslash X/U\times_H G$.
\end{Pro}
\begin{proof}
	Denote by $Y:={\bm p}^{-1}$. Note that each $U\times U$-orbit in $X$ intersects $Y$ at exactly one point. Thus we have the following and commuting diagram
  \[
    Y\xrightarrow{\sim} U\backslash X/U; \quad
    \begin{tikzcd}
      Y \arrow[r] \arrow[d] & X \arrow[d] \\
      H \arrow[r] & G
    \end{tikzcd}\ . 
  \]
  Then $X\cong Y\times_H G\cong U\backslash X/U\times_H G$.
\end{proof}

In what follows, we will describe a `trivialization' of $G^{(2)}$ and extend it to $G^{(n)}$ in the next section. Define a rational map $\pi$ on $G^{(2)}$ as
\begin{equation}\label{factors}
  \pi\colon G^{(2)}\to U \ : \ (g_1, g_2)\mapsto v_1u_2, \quad \text{~where~} g_1=u_1h_1\overline{w_0}v_1, \quad g_2=u_2h_2\overline{w_0}v_2.
\end{equation}
\begin{Pro}\cite[Proposition 2.42]{BKII}
  The following map is a birational isomorphism of varieties:
  \begin{equation}\label{trivilization}
    F\colon G^{(2)} \to M^{(2)}\times_H G \ :\ (g_1,g_2)\mapsto \left( \pi(g_1,g_2), \hw(g_1), \hw(g_2); g_1g_2\right),
  \end{equation}
  where the fiber product $M^{(2)}\times_H G=(U\times H^2)\times_H G$ is over
  \[
      \hw_2\colon (u,h_1,h_2)\mapsto \hw(\overline{w_0}u\overline{w_0})h_1h_2, \text{~and~} \hw\colon g\mapsto \hw(g).
  \]
\end{Pro}
Since $F$ is an $U\times U$-invariant isomorphism, we get an isomorphism of affine varieties
\begin{equation}\label{eq:identification}
  \overline{F}\colon U\backslash G^{(2)}/U \to M^{(2)}.
\end{equation}
Thus the central charge $\Delta_2$ on $G^{(2)}$ descends to a function on $M^{(2)}$:
\begin{equation}
   \overline{\Delta}_2:=\Delta_2\circ \overline{F}^{-1}.
\end{equation}
The following corollary is clear:
\begin{Cor}
  The isomorphism $F$ defined in \eqref{trivilization} is an isomorphism of varieties with potential:
  \[
    F\colon \left(G^{(2)}, \Phi_{G^{(2)}} \right) \to \left( M^{(2)}\times_H G, \overline{\Delta}_2+\Phi_{BK}\right).
  \] 
\end{Cor}
From now on, we refer to \eqref{trivilization} as a trivialization of $G^{(2)}$.

\subsection{The model space \texorpdfstring{$M^{(2)}$}{M{(2)}} }

In what follows, we will look in more details at the variety $M^{(2)}$, which is the key to the category of geometric multiplicities.

Denote by $M^{(n)}:=U^{n-1}\times H^n$ for $n\geqslant 2$. Note that by the usage of \eqref{trivilization}, for each vertex $a$ of associahedron $K_n$ for any $n\geqslant 3$, one can get a trivialization $F_a\colon G^{(n)}\to M^{(n)}\times_H G$.  For $n=3$, we have the
\[
   F_{12,3}\colon (G*G)*G\xrightarrow{F\times \id_G} \left(M^{(2)}\times_H G\right)* G\xrightarrow{\id_{M^{(2)}}\times F}M^{(2)}\times_H\left( M^{(2)}\times_H G\right) \xrightarrow{\sim}M^{(3)}\times_{H} G,
\]
and
\[
   F_{1,23}\colon G*(G*G)\xrightarrow{\id_G\times F} G*\left(G\times_H M^{(2)}\right)\xrightarrow{F\times \id_{M^{(2)}}} \left( G\times_H M^{(2)} \right)\times_H M^{(2)} \xrightarrow{\sim}M^{(3)}\times_{H} G, 
\]
where the fiber product $M^{(3)}\times_{H} G=(U^2\times H^3)\times_{H} G$ is over
\[
    \hw_{3}\colon (u_1,u_2,h_1,h_2,h_3)\mapsto \hw(\overline{w_0}u_1\overline{w_0})\hw(\overline{w_0}u_2\overline{w_0})h_1h_2h_3, \text{~and~} \hw\colon g\mapsto \hw(g).
\]
Hence we get the a non-trivial automorphism of $U\times U$-varieties
\[
  F_{1,23}\circ F_{12,3}^{-1}\colon M^{(3)}\times_{H} G\xrightarrow{\sim} M^{(3)}\times_{H} G. 
\]
Denote by $\pr_{M^{(3)}}$ and $\pr_G$ the natural projection of $ M^{(3)}\times_{H} G$ to $M^{(3)}$ and $G$ respectively. Then we see immediately that $\pr_G\circ F_{1,23}\circ F_{12,3}^{-1}(m,g)=g$. Note that $F_{12,3}$ and $F_{1,23}$ are $(U,U)$-equivariant, thus they descents to maps $\overline{F}_{1,23}$ and $\overline{F}_{12,3}$ respectively. By Proposition \ref{Thm5}, we have the following birational isomorphism:
\[
	\iota_{M^{(3)}}\colon M^{(3)} \to M^{(3)}\times_{H} U\backslash G/U\ :\ m\mapsto \left(m, U\hw_{3}(m)\overline{w_0}U\right).
\]
Assembling all these components, we get:
\[
  \Psi\colon M^{(3)}\xrightarrow{\iota_{M^{(3)}}} M^{(3)}\times_{H} U\backslash G/U \xrightarrow{\overline{F}_{12,3}^{-1}}U\backslash G^{(3)}/U\xrightarrow{\overline{F}_{1,23}} M^{(3)}\times_{H} U\backslash G/U \xrightarrow{\pr_{M^{(3)}}} M^{(3)}.
\]
\begin{Pro}\label{formforasso}
  The map $\Psi$ defined above is an automorphism of variety $U^2\times H^3$:
  \[
    \Psi=\pr_{M^{(3)}}\circ \overline{F}_{1,23}\circ \overline{F}_{12,3}^{-1}\circ\iota_{M^{(3)}}\colon U^2\times H^3\to U^2\times H^3.
  \]
\end{Pro}

\begin{Ex}\label{mainEx}
	Here we work out one example of the map $\Psi$ for $G=\GL_2$. Write
	\[
		x_i:=\begin{bmatrix}
			a_i & 0\\
			b_i & c_i
		\end{bmatrix}
	\]
	as coordinates for $B_-$. Suppose that $b_i\in \mathbb{G}_{\mathbf{m}}$, then we have
	\[
		\begin{bmatrix}
			a_i & 0\\
			b_i & c_i
		\end{bmatrix}=
		\begin{bmatrix}
			1 & \dfrac{a_i}{b_i}\\
			0 & 1
		\end{bmatrix}
		\begin{bmatrix}
			\dfrac{a_ic_i}{b_i} & 0\\
			0 & b_i
		\end{bmatrix}
		\begin{bmatrix}
			0 & -1\\
			1 & 0
		\end{bmatrix}
		\begin{bmatrix}
			1 & \dfrac{c_i}{b_i}\\
			0 & 1
		\end{bmatrix},
	\]
	Note the map $F_{12,3}$ restrict to $(B_-\times B_-)\times B_-\xrightarrow{\sim} M^{(3)}\times_H B_-$, which sends $(x_1,x_2,x_3)$ to 
	\[
		\left(
		\begin{bmatrix}
			1 & \dfrac{1}{X_{12}b_1b_2}\\
			0 & 1
		\end{bmatrix},
		\begin{bmatrix}
			1 & \dfrac{X_{12}}{b_3Y}\\
			0 & 1
		\end{bmatrix},
		\begin{bmatrix}
			\dfrac{a_1c_1}{b_1} & 0\\
			0 & b_1
		\end{bmatrix},
		\begin{bmatrix}
			\dfrac{a_2c_2}{b_2} & 0\\
			0 & b_2
		\end{bmatrix},
		\begin{bmatrix}
			\dfrac{a_3c_3}{b_3} & 0\\
			0 & b_3
		\end{bmatrix},
		x_1x_2x_3
		\right), 
	\]
	where $X_{ij}^{-1}=b_ia_j+c_ib_j$ and $Y^{-1}=b_1a_2a_3+c_1b_2a_3+c_1c_2b_3$. For elements in $M^{(3)}\times_H B_-$, we use 
	\[
		\left(
		\begin{bmatrix}
			1 & u_1\\
			0 & 1
		\end{bmatrix},
		\begin{bmatrix}
			1 & u_2\\
			0 & 1
		\end{bmatrix},
		\begin{bmatrix}
			e_1 & 0\\
			0 & f_1
		\end{bmatrix},
		\begin{bmatrix}
			e_2 & 0\\
			0 & f_2
		\end{bmatrix},
		\begin{bmatrix}
			e_3 & 0\\
			0 & f_3
		\end{bmatrix},
		\begin{bmatrix}
			p & 0\\
			u_1u_2\prod f_i & p^{-1}\prod e_i \prod f_i
		\end{bmatrix}
		\right)
	\]
	as coordinates. Thus the inverse map of $F_{12,3}$ is given by
	\begin{align*}
		c_3&=\frac{u_1e_3+p^{-1}f_3\prod e_i}{u_1u_2};&  a_3&=e_3f_3c_3^{-1};&  b_3&=f_3;\\
		c_2&=\frac{e_2}{u_1}(1+p^{-1}e_1f_2a_3);& a_2&=e_2f_2c_2^{-1}; &  b_2&=f_2;\\
		a_1&=\frac{e_1}{u_1}(1+pe_1^{-1}f_2^{-1}a_3^{-1});&  c_1&=e_1f_1a_1^{-1}; &  b_1&=f_1.
	\end{align*}
	Then applying $F_{1,23}$, we get
	\[
		\left(
		\begin{bmatrix}
			1 & \dfrac{e_2}{u_1f_2}+u_2\\
			0 & 1
		\end{bmatrix},
		\begin{bmatrix}
			1 & \dfrac{u_1u_2f_2}{u_1^{-1}e_2+f_2u_2}\\
			0 & 1
		\end{bmatrix},
		\begin{bmatrix}
			e_1 & 0\\
			0 & f_1
		\end{bmatrix},
		\begin{bmatrix}
			e_2 & 0\\
			0 & f_2
		\end{bmatrix},
		\begin{bmatrix}
			e_3 & 0\\
			0 & f_3
		\end{bmatrix},
		x_1x_2x_3
		\right).
	\]
	Thus the map $\Psi$ is just then map send 
	\[
		(u_1,u_2,e_i,f_i) \to \left(\dfrac{e_2}{u_1f_2}+u_2,\dfrac{u_1u_2f_2}{u_1^{-1}e_2+f_2u_2}, e_i,f_i\right).
	\]
\end{Ex}

\subsection{The category of geometric multiplicities}

Now we are ready to define the category $\mul_G$ of geometric multiplicities. 
\begin{Def}\label{mulg}
  For any semisimple group $G$, the category $\mul_G$ of geometric multiplicities is:
   \begin{itemize}
      \item The object in $\mul_G$ is a quadruple ${\bm M}=(M,\Phi_M,\hw_M,\pi_M)$, which consists of an irreducible affine variety $M$ with potential $\Phi$, a rational map $\hw_M\colon M\to H$ from $M$ to the Cartan subgroup $H$ of $G$ and a rational map $\pi_M\colon M\to S_M$ from $M$ to a split torus $S_M$;

      \item A morphism ${\bm f}\colon {\bm M}\to {\bm N}$ is a triple of rational maps $f_1\colon M\to N $, and $f_2\colon H\to H$ and $f_3\colon S_M\to S_N$ s.t. $\hw_N\circ f_1=f_2\circ \hw_M$ and $\pi_N\circ f_1=f_3\circ \pi_M$.
    \end{itemize}
    Moreover, we have a binary operation $\star$ in $\mul_G$:
    \[
       {\bm M}\star {\bm N}:= (M\star N, \Phi_{M\star N},\hw_{M\star N},\pi_{M\star N}),
    \]
    where each component is defined as follows:
    \begin{itemize}
      \item $M\star N:=(M\times N)\times_{H^2} M^{(2)}$, where the fiber product is over $\hw_M\times \hw_N$ and the  natural projection $\pr_{H^2}\colon M^{(2)}\to H^2$;

      \item $\Phi_{M\star N}(m,n; u, h_1, h_2)=\Phi_M(m)+\Phi_{N}(n)+\overline{\Delta}_2(u,h_1,h_2)$;

      \item $\hw_{M\star N}\colon M\star N\to  H\ : \ (m,n,u)\mapsto \hw_{M}(m)\hw_N(n)\hw(\overline{w_0}u\overline{w_0})$;

      \item $S_{M\star N}=S_M\times S_N\times H^2$ and 
      \[
        \pi_{M\star N} \colon M\star N\to S_M\times S_N\times H^2\ :\ (m,n,u)\mapsto (\pi_{M}(m), \pi_N(n), \hw_M(m), \hw_{N}(n)).
      \]
    \end{itemize}
\end{Def}

\begin{Ex}\label{H*H}
  The first example of objects in $\mul_G$ is just ${\bm H}:=(H, 0, \id,0)$. Then ${\bm H}\star {\bm H}$ is
  \[
    (M^{(2)}, \overline{\Delta}_2, \hw_{2}, \pi_{2}),
  \]
  where $\hw_{2}(u,h_1,h_2)=h_1h_2\hw(\overline{w_0}u\overline{w_0})$ and $\pi_{2}(u,h_1,h_2)=(h_1,h_2)$.
\end{Ex}

From the definition, one natural question to ask is whether the binary operation $\star$ is associative. First, by defintion, the triple product $(M_1\star M_2)\star M_3$ isomorphic to $(M_1\times M_2 \times M_3)\times_{H^3} M^{(3)}$ in the natural way, so is $M_1\star (M_2\star M_3)$. By Proposition \ref{formforasso}, we have the following non-trivial isomorphism, which will be our associator.
\begin{equation}\label{assoformula}
  \begin{tikzcd}
    \widetilde{\Psi}_{M_1,M_2,M_3}\colon (M_1\star M_2)\star M_3 \arrow[r, "\sim"]  & (M_1\times M_2 \times M_3)\times_{H^3} M^{(3)} \arrow[d, "\id\times \Psi"]  & \\[1pt]
    & (M_1\times M_2 \times M_3)\times_{H^3} M^{(3)} \arrow[r, "\sim"] & M_1\star (M_2\star M_3)
  \end{tikzcd}
\end{equation} 
\begin{Thm}\label{Associator}
  For a reductive group $G$, the category $\mul_G$ is  monoidal with product ${\bm M}_1\star {\bm M}_2$ given by Definition \ref{mulg}, and associator given by the formula \eqref{assoformula}.
\end{Thm}

Next, we will add positive structure to the category $\mul_G$.
\begin{Def}
	A geometric multiplicity $M\in\mul_G$ is positive if there exists a positive structure $\Theta_M$ on $M$ s.t. $(M, \Phi_M, \Theta_M)$ is a positive variety with potential, and $\hw_M$ (resp. $\pi_M$) is a $(\Theta_M, \Theta_H)$ (resp. $(\Theta_M, \Theta_S)$) positive map. For simplicity, we denote by ${\bm M}$ the quintuple $(M,\Phi_M, \Theta_M, \hw_M, \pi_M)$ as well. A morphism ${\bm f}=(f_1,f_2,f_3)$ of positive geometric multiplicities ${\bm M}$ and ${\bm N}$ is morphism of geometric multiplicities s.t. $f_1\colon (M,\Phi_M,\Theta_M)\to (N,\Phi_N\Theta_M)$ is morphism of positive varieties, $f_2$ and $f_2$ are positive maps of tori.  Denote by $\mul_G^+$ the subcategory of $\mul_G$ consists of positive geometric multiplicities. 	
\end{Def} 

\begin{Pro}\label{posofg2}
  Denote by $\Theta_{M^{(2)}}:=\Theta_U\times \Theta_H \times \Theta_H$ the positive structure on $M^{(2)}$. Then 
  \[
    {\bm H}\star {\bm H}=\left(M^{(2)}, \overline{\Delta}_2, \hw_{2}, \pi_{2}\right)\in \mul_G
  \]
  is a positive geometric multiplicity.
\end{Pro}
\begin{proof}
  Let us consider the toric chart $\Theta_{M^{(2)}}:=\Theta_U\times \Theta_H\times \Theta_H$ for $M^{(2)}$. What we will show actually is that $(M^{(2)},\overline{\Delta}_2,\Theta_{M^{(2)}})$ is a positive variety with potential. Denote by $g_1=u_1h_1\overline{w_0}v_1, g_2=u_2h_2\overline{w_0}v_2$.

  For $(u,h_1,h_2)\in U\times H^2$, choose a lift in $G\times G$ as $(h_1\overline{w_0}u, h_2\overline{w_0})$. Then one has:
  \begin{align*}
    \overline{\Delta}_{2}(u,h_1,h_2)&=\Delta_{2}\circ \overline{F}^{-1}(u,h_1,h_2)\\
    &=\Phi_{BK}(h_1\overline{w_0}u)+\Phi_{BK}(h_2\overline{w_0})-\Phi_{BK}(h_1\overline{w_0}uh_2\overline{w_0})\\
    &=\chi^{\st}(u)-\Phi_{BK}(h_1\overline{w_0}uh_2\overline{w_0})\\
    &=\chi^{\st}(u)+\Phi_{BK}(h_2u^Th_1^{w_0}),
  \end{align*}
  where $h_1^{w_0}$ is short for $\overline{w_0}^{-1}h_1\overline{w_0}$ for $x\in G$ and the last equality is because that for $g=u\overline{w_0}hv$
  \[
    \Phi_{BK}(g)=\chi^{\st}(u)+\chi^{\st}(v)=-\chi^{\st}\left(\left(u^T\right)^{w_0}\right)-\chi^{\st}\left(\left(v^T\right)^{w_0}\right)=-\Phi_{BK}\left(\left(g^T\right)^{w_0}\right).
  \]
  Then we see immediately that the function $\overline{\Delta}_2$ is positive with respect to the positive structure $\Theta_{M^{(2)}}$.
\end{proof}

\begin{Ex}\label{H*H2}
  In the rest of the paper, denote by
  \[
    (M^{(n)}, \overline{\Delta}_n, \hw_{n}, \pi_{n}):=(\cdots(({\bm H}\star{\bm H})\star {\bm H})\star\cdots \star {\bm H})
  \]
  the $\star$ products of $n$ copies of ${\bm H}$  in the canonical order. Denote by ${\bm u}:=(u_1,\ldots,u_{n-1},h_1,\ldots,h_n)$. Then the potential $\overline{\Delta}_n$ is given by
  \[
  	\overline{\Delta}_n({\bm u})=\sum_{i=1}^{n-1} \chi^{\st}(u_i)+\sum \Phi_{BK}(h_{i+1}u_ip_i^{w_0}({\bm u})),
  \]
  where $p_i({\bm u})=h_i\prod_{j=1}^{i}\hw(\overline{w_0}u_i\overline{w_0})h_j$. Note that this potential can be interpreted as the decent of the central charge $\Delta_n$ of $G^{(n)}$ under the canonical trivialization $G^{(n)}\cong M^{(n)}\times_H G$. We leave it as an excise for the reader to write down explicit formulas for $\hw_{n}$ and $\pi_{n}$. Then we know $(M^{(n)}, \overline{\Delta}_n, \hw_{n}, \pi_{n})$ is a positive geometric multiplicity.
\end{Ex}

As a corollary of this proposition, the binary product $\star$ is well defined in the category $\mul_G^+$. Moreover:

\begin{Main}\label{positiveasso}
  For any semi-simple algebraic group $G$,  the category $\mul_G^+$ is  monoidal with product ${\bm M}_1\star {\bm M}_2$ given by Definition \ref{mulg}, and associator given by the formula \eqref{assoformula}.
\end{Main}
\begin{Rmk}\label{rmkofposass}
  Because of Theorem \ref{Associator} and Proposition \ref{posofg2}, what Theorem \ref{positiveasso} says is that the associator $\widetilde{\Psi}_{M_1,M_2,M_3}$ in \eqref{assoformula} is a positive isomorphism of positive varieties.  
\end{Rmk}

\subsection{Higher central charges}\label{hcc}

Actually, there are more $U\times U$-invariant positive functions on $G^{(n)}$, which we call {\em higher central charges}. First, let us denote by
\[
  \Phi^l(uh\overline{w_0}u')=\chi^{\st}(u),\quad \Phi^r(uh\overline{w_0}u')=\chi^{\st}(u'), \quad \text{for~} uh\overline{w_0}u'\in G^{w_0},
\]
the two component of $\Phi_{BK}$ on $G$. It is easy to see
\[
	\Phi^l(ugu')=\chi^{\st}(u)+\Phi^l(g),\quad \Phi^r(ugu')=\Phi^r(g)+\chi^{\st}(u'), \quad \text{for~} u,u'\in U \text{~and~} g\in G^{w_0}.
\]
\begin{Lem}\label{ccchc}
	The following three functions on $G*G$ are positive and $U\times U$-invariant: for $g_1*g_2\in G*G$
	\[
		c_0(g_1*g_2)=\Phi^r(g_1)+\Phi^l(g_2), c_1(g_1*g_2)=\Phi^l(g_1)-\Phi^l(g_1g_2), c_2(g_1*g_2)=\Phi^r(g_2)-\Phi^r(g_1g_2).
	\]
\end{Lem}
\begin{proof}
  First of all, one need to check $c_i$'s are well defined. We only need to prove the positivity $c_1$ and $c_2$. Recall that one can rewrite \cite[Lemma 1.24]{BKII}
  \[
    \Phi^l(g)=\chi^{\text{st}}([\overline{w_0}^{-1}g^\iota]_+), \quad \Phi^r(g)=\chi^{\text{st}}([\overline{w_0}^{-1}g]_+).
  \]
  Denote by $g_i=u_ih_i\overline{w_0}v_i$, then one compute
  \[
    \Phi^r(g_1)-\Phi^r(g_1g_2)=\chi^{\st}(v_2)-\chi^{\st}([v_1u_2h_2\overline{w_0}]_+)-\chi^{\st}(v_2).
  \]
  Note that $[g^{-\iota}]_+=[g]_+^{-\iota}$. Then
  \[
    \chi^{\text{st}}([g^{-\iota}]_+)=-\chi^{\text{st}}([g]_+).
  \]
  Let $b=v_1u_2t'$. Then we have
  \[
    -\chi^{\text{st}}([b\overline{w_0}]_+)=\chi^{\text{st}}([b^{-\iota}\overline{w_0}]_+)=\chi^{\text{st}}([b^{-\iota}\overline{w_0}^{-1}]_+)=\chi^{\text{st}}([\overline{w_0}^{-1}\sigma(b)]_+)
  \]
  By \cite[Theorem 4.7]{BZ01}, the maps $b\mapsto [\overline{w_0}^{-1}b]_+$ and $\sigma$ are positive. Thus we know $\Phi^r(g)-\Phi^r(fg)$ is positive. It is similar to show that $\Phi^l(g_1)-\Phi^l(g_1g_2)$ is positive.
\end{proof}
\begin{Rmk}
	Note that the central charge $\Delta_2$ we considered before can be written as $\Delta_2=c_0+c_1+c_2$.
\end{Rmk}

Recall we have the notion of $U\times U$-varieties, see Definition \ref{uuvariety}. Beside the central charge we defined in \eqref{eq:Delta}, one can get more $U\times U$ invariant function by using $U\times U$ equivariant maps. To be more precise, given a  $U\times U$-equivariant map $f\colon X\to Y$ of $U\times U$-varieties $X$ and $Y$, any $U\times U$-invariant function $c$ on $Y$ can be lifted to a $U\times U$-invariant function $c\circ f$ on $X$.


To get $U\times U$-invariant positive functions on $G^{(n)}$, we use the following natural $U\times U$-equivariant positive projections:
\[
	m_{k}^n\colon G^{(n)}\to G^{(n-1)} \ :\ g_1*\cdots* g_n\to g_1*\cdots*(g_kg_{k+1})*\cdots *g_n
\]
and maps
\begin{align*}
	p_n&\colon G^{(n)}\to U\backslash G^{(n-1)} \ :\ g_1*\cdots* g_n\to [g_2*\ldots *g_n];\\
	q_n&\colon G^{(n)}\to G^{(n-1)}/U \ :\ g_1*\cdots* g_n\to [g_1*\ldots *g_{n-1}].
\end{align*}
If a function $\Delta$ on $G^{(n)}$ is $U\times U$-invariant, then $\Delta$ descends to functions on $G^{(n-1)}/U$ and $U\backslash G^{(n-1)}$, which will be denoted by $\Delta$ as well by abusing of notation.
\begin{Pro}\label{uuinv}
	For any $U\times U$-invariant positive function $\Delta$ on $G^{(n-1)}$, the following functions
	\[
		\Delta\circ m_k^n, \quad, \Delta\circ p_n, \quad \Delta\circ q_n \quad \text{~for~} k\in [1,n-1]
	\]
	are $U\times U$-invariant positive on $G^{(n)}$. Denote by ${\bf C}_n$ the set of these functions, and we call them central charges.
\end{Pro}

\begin{Ex}\label{ex:n=3}
	The set ${\bf C}_3$ for any $G$ is:
	\begin{align*}
		c_0\circ m_1^3&\colon \Phi^r(g_1g_2)+\Phi^l(g_3);& c_1\circ m_1^3&\colon \Phi^l(g_1g_2)-\Phi^l(g_1g_2g_3);&  c_2\circ m_1^3&\colon \Phi^r(g_3)-\Phi^r(g_1g_2g_3);&\\
		c_0\circ m_2^3&\colon \Phi^r(g_1)+\Phi^l(g_2g_3);&  c_1\circ m_2^3&\colon \Phi^l(g_1)-\Phi^l(g_1g_2g_3);& c_2\circ m_2^3&\colon \Phi^r(g_2g_3)-\Phi^r(g_1g_2g_3);&\\
		c_0\circ p_3&\colon \Phi^r(g_2)+\Phi^l(g_3);& c_0\circ q_3&\colon \Phi^r(g_1)+\Phi^l(g_2). & &
	\end{align*}
	Note that functions $c_i\circ p_3$ and $c_i\circ q_3$ for $i=1,2$ are not in the list because they are the linear combinations of the other. For example, we have:
	\[
		c_1\circ q_3=c_1\circ m_2^3-c_1\circ m_1^3.
	\]
	Also one can check the 8 functions in the list are linear independent. Moreover one can check
	\[
	    \sum_{i=1}^3\Phi(g_i)-\Phi(g_1g_2g_3)=c_2\circ m_1^3+c_1\circ m_2^3+c_0\circ p_3+c_0\circ p_3.
	\]
	which gives us the central charges on $G^{(3)}$.
\end{Ex}
Since the higher central charges are $U\times U$ invariant, they descend to positive functions on $B^-\times \cdots \times B^-$. In the following, let us consider $G=\GL_2$ and $n=3$ and we follow the notation in Example \ref{mainEx}.
\begin{Ex}\label{hccGL2}
	 Note that $\Phi^l(x_i)=a_ib_i^{-1}$ and $\Phi^r(x_i)=(a_ib_i)^{-1}$. Direct computation gives ${\bf C}_3(\GL_2)$
	\begin{align*}
		c_0\circ m_1^3&\colon b_3^{-1}X_{12}Y^{-1};& c_1\circ m_1^3&\colon a_1a_2c_1c_2b_3X_{12}Y;&  c_2\circ m_1^3&\colon a_3b_3^{-1}c_3X_{12}^{-1}Y;&\\
		c_0\circ m_2^3&\colon b_1^{-1}X_{23}Y^{-1};&  c_1\circ m_2^3&\colon a_1b_1^{-1}c_1X_{23}^{-1}Y;& c_2\circ m_2^3&\colon a_2a_3c_2c_3b_1X_{23}Y;&\\
		c_0\circ p_3&\colon b_2^{-1}b_3^{-1}X_{23}^{-1};& c_0\circ q_3&\colon b_1^{-1}b_2^{-1}X_{12}^{-1}. &&
	\end{align*}
	where
	\[
		X_{ij}^{-1}=b_ia_j+c_ib_j; \quad Y^{-1}=b_1a_2a_3+c_1b_2a_3+c_1c_2b_3.
	\]
	Next we will write these functions using coordinates on our geometric multiplicity space. Recall the map $F_{12,3}$ and its inverse in Example \ref{mainEx}, then we can write all these functions on $M^{(2)}$ as 
	\begin{align*}
		c_0\circ m_1^3&\colon u_2;& c_1\circ m_1^3&\colon \dfrac{e_1e_2}{f_1f_2}u_1^{-2}u_2^{-1};&  c_2\circ m_1^3&\colon \dfrac{e_3}{f_3}u_2^{-1};&\\
		c_0\circ m_2^3&\colon u_2';&  c_1\circ m_2^3&\colon \dfrac{e_1}{f_1}\cdot\dfrac{1}{u_2'};& c_2\circ m_2^3&\colon \dfrac{e_2e_3}{f_2f_3}(u_1')^{-2}(u_2')^{-1};&\\
		c_0\circ p_3&\colon u_1';& c_0\circ p_3&\colon u_1. &&
	\end{align*}
	where we recall
	\[
		u_1'=\dfrac{e_2}{u_1f_2}+u_2; \quad u_2'=\dfrac{u_1u_2f_2}{u_1^{-1}e_2+f_2u_2}.
	\]
\end{Ex}

\subsection{Isomorphism of geometric multiplicities}

In this section, we show one example of isomorphism of geometric multiplicities arising from Howe duality without giving all details. This section is motivated by Howe $(\GL_2, \GL_n)$-duality \cite{Howe}. Let $\lambda$ be a Young diagram of depth $\leqslant n$ and $V_\lambda^n$ be the polynomial representation of $\GL_n$ parametrized by $\lambda$. The following is clear:
\[
	\bigoplus_{(p_1,\ldots,p_n)}S^{p_1}(\mathbb{C}^2)\otimes\cdots \otimes S^{p_l}(\mathbb{C}^2)\cong \mathbb{C}
	\begin{bmatrix}
		x_n & x_{n-1} & \ldots & x_1\\
		y_n & y_{n-1} & \ldots & y_1
	\end{bmatrix}\cong
	\bigoplus_{\text{depth}(\lambda)\leqslant 2} V_\lambda^2\otimes V_\lambda^n,
\]
where $S^l(\mathbb{C}^m)$ is the symmetric polynomial of degree $l$. Moreover we have
\begin{equation}\label{eq:eq3}
	S^{p_1}(\mathbb{C}^2)\otimes\cdots \otimes S^{p_n}(\mathbb{C}^2) \cong \bigoplus V_\lambda^l(p_1,\ldots,p_n)\otimes V_\lambda^2.	
\end{equation}
Thus for the tensor multiplicities, we have $\left[V_\lambda^n \ : \ S^{p_1}(\mathbb{C}^n)\otimes\cdots \otimes S^{p_l}(\mathbb{C}^n) \right]=\dim V_\lambda^l(p_1,\ldots,p_l)$. The following argument is to a way to relate our geometric multiplicity $M^{(n)}$ to a subvariety of $\mathbb{C}_\mathbf{m}^{2n}$.

We consider {\em geometric (pre)-crystals} of $\GL_2$:
\[
	{\bm X}:=\left( \mathbb{G}_{\mathbf{m}}^2, \gamma, \varphi, \varepsilon, e^{\sbt} \right),
\]
where for $(x,y)\in \mathbb{G}_{\mathbf{m}}^2$,
\[
	\gamma(x,y)=\begin{bmatrix}
		x & 0\\
		0 & y
	\end{bmatrix},\  \varphi(x,y)=x^{-1},\  \varepsilon(x,y)=y^{-1},\ e^c(x,y)=(cx, c^{-1}y).
\]
Note that ${\bm X}$ admits a natural potential $\Phi(x,y)=x+y$. By \cite[Definition 2.15]{BKII}, we consider the $n$ copies of ${\bm X}$. Write the basis variety of ${\bm X}^n:=\left( \mathbb{G}_{\mathbf{m}}^{2n}, \gamma_n, \varphi_n, \varepsilon_n, e^{\sbt}_n \right)$ as $2\times n$ matrix with the following coordinates:
\[
	{\bm x}:=\begin{bmatrix}
		x_n & x_{n-1} &\cdots & x_1\\
		y_n & y_{n-1} &\cdots & y_1
	\end{bmatrix}
\]
with potential $\Phi_n=\sum_1^n (x_i+y_i)$. Direct computation shows that
\[
	\varepsilon_n({\bm x})=\sum_{i=0}^{n-1} \frac{x_1\cdots x_i}{y_1\cdots y_{i+1}}=\frac{1}{y_1}\left(1+\sum_{i=1}^{n-1} \frac{x_1\cdots x_i}{y_2\cdots y_{i+1}}\right).
\]

Let $H:=\{\diag(x,y)\mid x,y\in \mathbb{G}_{\mathbf{m}}\}$ be the Cartan subgroup of $\GL_2$ and
\[
	T:=\{\diag(x,1) \mid x\in \mathbb{G}_{\mathbf{m}}\}\subset H
\]
be a subgroup of $H$. Consider the subvariety $Y$ of $\mathbb{G}_\mathbf{m}^{2n}$ given by $y_1=1$. Then we have the following positive geometric multiplicity:
\[
	{\bm M}_2^n:=\left(Y, \Phi_n+\varepsilon_n, \id_Y \gamma_n, p_n\right),
\]
where $p_n\colon {\bm M}_2^n\to T^n$ by sending ${\bm x}$ to $\diag(x_1,1)\times \diag(x_ny_n,1)$ $\id_Y$ is the positive structure on $Y$. Recall we have a positive geometric multiplicity $(U^{n-1}\times H^n, \overline{\Delta}_n, \hw_n, \pi_n)$ for $\GL_2$, which restrict to
\[
	{\bm U}_2^n:=(U^{n-1}\times T^n, \overline{\Delta}_n, \Theta_U^{n-1}\times \theta_T^{n}, \hw_n, \pi_n).
\]
Denote by ${\bm u}:=(u_1,\ldots,u_{n-1}; e_1,\ldots, e_n)$ the natural coordinates on ${\bm U}_2^n$. Then
\begin{Thm}
	The following map $\varPsi\colon {\bm M}_2^n\to {\bm U}_2^n$ is a morphism of positive geometric multiplicities:
	\[
		u_i\circ \varPsi= y_{i+1}; \quad e_i\circ \varPsi = x_iy_i, 	
	\] 
	{\em i.e.}, we have $\hw_n\circ \varPsi=\gamma_n$, $\pi_n\circ \varPsi=p_n$ and $\Phi_n+\varepsilon_n-\overline{\Delta}_n\circ \varPsi$ is positive. Moreover,
	\begin{equation}\label{compar}
		\dim \mathbb{C}\left[ (\gamma_n, p_n)^{-t}(\mu; \lambda_1,\ldots,\lambda_n)\right]= \dim \mathbb{C}\left[ (\hw_n, \pi_n)^{-t}(\mu; \lambda_1,\ldots,\lambda_n)\right],
	\end{equation}
	where $\mu\in X_*(H)$ and $\lambda_i\in X_*(T)$.
\end{Thm}
\begin{proof}
	Note that $\hw_n\circ \varPsi=\gamma_n$ and  $\pi_n\circ \varPsi=p_n$ is clear from the definition. Recall that $\overline{\Delta}_n$ is given by
	\[
		\overline{\Delta}_n({\bm u})=\sum_{i=1}^{n-1} \left(u_i+\frac{e_{i+1}}{u_i}+\frac{e_1\cdots e_{i}}{(u_1\cdots u_{i-1})^2}\cdot\frac{1}{u_i}\right).
	\]
	Direct computation shows
	\[
		\frac{e_{i+1}}{u_i}=x_{i+1}; \quad \frac{e_1\cdots e_{i}}{(t_1\cdots u_{i-1})^2}\cdot\frac{1}{u_i}=\frac{x_1\cdots x_i}{y_2\cdots y_{i+1}}.
	\]
	Thus $\Phi_n+\varepsilon_n-\overline{\Delta}_n\circ \varPsi=x_1+2$. It is easy to see that
	\[
		(\Phi_n+\varepsilon_n)^t=(\Phi_n+\varepsilon_n-x_1-2)^t.
	\]
	Thus we get the comparison \eqref{compar}.
\end{proof}

We hope to generalize the result here to Howe $(\GL_m, \GL_n)$-duality in future work.

\section{From Positive Geometric Multiplicities to Tensor Multiplicities}\label{tensor mul}

In this section, we will explain how pass from the geometric multiplicities to tensor multiplicities by the usage of tropicalization.

For ${\bm M}=(M,\Phi_M, \Theta_M, \hw_M, \pi_M)\in \mul_G^+$, the tropicalization of the positive map $\pi_M\times \hw_M$ gives a morphism of affine tropical varieties:
\[
  (\pi_M\times \hw_M)^t\colon (M, \Phi_M)^t\to S^t\times H^t=X_*(S)\times X_*(H),
\]
where $X_*(S)$ (resp. $X_*(H)$) is the cocharacter lattice of the torus $S$ (resp. $H$). Note that $X_*(H)$  is isomorphic naturally to the character lattice $X^*(H^\vee)$ of $G^\vee$.  For $(\xi, \lambda^\vee)\in X_*(S)\times X_*(H)$, denote by
\[
  M^t_{\xi,\lambda^\vee}:=(\pi_M\times \hw_M)^{-t}(\xi, \lambda^\vee)
\]
the tropical fiber of $(M, \Phi_M)^t$. We say the positive geometric multiplicity ${\bm M}$ is {\em finite} if the morphism $(\pi_M\times\hw_M)^t$ is finite as in Definition \ref{finite}.

To each positive geometric multiplicity ${\bm M}$, we assign a $G^\vee$-module $\mathcal{V}({\bm M})$ to ${\bm M}$ via
\[
  \mathcal{V}({\bm M})=\bigoplus_{(\xi,\lambda^\vee)\in X_*(S)\times X_*^+(H)} \mathbb{C}[M^t_{\xi,\lambda^\vee}]\otimes V_{\lambda^\vee},
\]
where $X_*^+(H)$ is the set of dominant integral coweights of $G$ and $V_{\lambda^\vee}$ is the irreducible representation of $G^\vee$ with highest weight $\lambda^\vee$. Denote by $\mathbf{Mod}_{G^\vee}$ the category of $G^\vee$-modules.
\begin{Main}\label{functor}
  The assignments ${\bm M}\mapsto \mathcal{V}({\bm M})$ is a monoidal functor from $\mul_G^+$ to $\mathbf{Mod}_{G^\vee}$.
\end{Main}

For $\xi\in X_*(S)$,  define the typical $\xi$-component $\mathcal{V}_{\xi}({\bm M})$ by
\begin{equation}\label{typcom}
  \mathcal{V}_{\xi}({\bm M}):=\bigoplus_{\lambda^\vee\in X_*^+(H)} \mathbb{C}[M^t_{\xi,\lambda^\vee}]\otimes V_{\lambda^\vee}.
\end{equation}
Then we have a similar statement as Theorem \ref{functor} for the typical components: 
\begin{Thm}\label{functor-compo}
  Given positive geometric multiplicities ${\bm M}_i$ for $i=1,2$, one has the following natural isomorphism of $G^\vee$-modules
  \begin{equation}\label{intro-eq1}
    \mathcal{V}_{\xi_1,\xi_2,\lambda^\vee, \nu^\vee}({\bm M}_1\star {\bm M}_2) \cong I_{\lambda^\vee}\left(\mathcal{V}_{\xi_1}({\bm M}_1)\right)\otimes I_{\nu^\vee}\left(\mathcal{V}_{\xi_1}({\bm M}_2)\right),
  \end{equation}
  where $I_{\lambda^\vee}(V)$ denotes the $\lambda^\vee$-th isotypic component of a $G^\vee$-module $V$.
\end{Thm}

One of the fundamental problems of the representation of $G$ is to determine the tensor product multiplicity $c_{\lambda,\nu}^\mu$ of $V_\mu$ in $V_\lambda\otimes V_\nu$. Now we can find a solution to this problem by using Theorem \ref{functor-compo}:
\begin{Thm}\label{multi-variety}
  The positive geometric multiplicity $(M^{(2)}, \overline{\Delta}_2, \hw_{2}, \pi_{2})$ is finite and the tensor multiplicity $c_{\lambda^\vee, \nu^\vee}^{\mu^\vee}$ is the multiplicity of $(\lambda^\vee, \nu^\vee, \mu^\vee)$ over $(\pi_2\times \hw_2)^t$, {\em i.e.},
  \[
    c_{\lambda^\vee, \nu^\vee}^{\mu^\vee}=\dim \mathbb{C}\left[\left(M^{(2)}\right)^t_{\lambda^\vee,\nu^\vee,\mu^\vee}\right].
  \]
\end{Thm}
\begin{proof}
  Note that ${\bm H}:=(H, 0, \id, 0)$ is an object in the category $\mul_G^+$. By definition, 
  \[
    \mathcal{V}({\bm H})=\bigoplus_{\lambda^\vee\in X_*^+(H)} V_{\lambda^\vee}.
  \]
  By Proposition \ref{posofg2}, the geometric multiplicity ${\bm H}\star {\bm H}=(M^{(2)}, \overline{\Delta}_2, \hw_{2}, \pi_{2})$ is positive. Applying Theorem \ref{functor-compo} to ${\bm H}\star {\bm H}$: 
  \[
    \mathcal{V}_{\lambda^\vee,\nu^\vee}({\bm H}\star {\bm H})\cong \mathcal{V}_{\lambda^\vee}({\bm H})\otimes \mathcal{V}_{\nu^\vee}({\bm H})=V_{\lambda^\vee}\otimes V_{\nu^\vee}. 
  \]
  Together with the definition \eqref{typcom} of $\mathcal{V}_{\lambda^\vee,\nu^\vee}({\bm H}\star {\bm H})$, one gets:
  \[
    V_{\lambda^\vee}\otimes V_{\nu^\vee}\cong \mathcal{V}_{\lambda^\vee,\nu^\vee}({\bm H}\star {\bm H})=\bigoplus_{\lambda\in X_*^+(H)} \mathbb{C}\left[\left(M^{(2)}\right)^t_{\lambda^\vee,\nu^\vee,\mu^\vee}\right]\otimes V_{\mu^\vee},
  \]
  which gives the statement we need.
\end{proof}
Similarly, for $n\geqslant 2$, denote by $c_{\lambda_1,\dots,\lambda_n}^\mu$ the higher tensor multiplicities:
\[
   \bigotimes_{i=1}^{n} V_{\lambda_i}=\bigoplus_{\mu} c_{\lambda_1,\dots,\lambda_n}^\mu V_{\mu}.
\]
\begin{Thm}\label{n-multi-variety}
  For $n\geqslant 2$, the positive geometric multiplicity $(M^{(n)}, \overline{\Delta}_n, \hw_n,\pi_n)$ is finite and the tensor multiplicity $c_{\lambda_1^\vee,\dots,\lambda_n^\vee}^{\mu^\vee}$ the multiplicity of $(\lambda_1^\vee,\dots,\lambda_n^\vee,\mu^\vee)$ over $(\pi_n\times \hw_n)^t$, {\em i.e.},
  \[
    c_{\lambda_1^\vee,\dots,\lambda_n^\vee}^{\mu^\vee}= \dim \mathbb{C}\left[\left(M^{(n)}\right)^t_{\lambda_1^\vee,\ldots,\lambda_{n}^\vee,\mu^\vee}\right].
  \]
\end{Thm}

\section{Combinatoric Expressions of Tensor Multiplicities}\label{combforten}

\subsection{Preliminary on representation theory}

Recall that the coordinate algebra $\mathbb{Q}[G]$ can be realized as certain subalgebra of algebra $U(\mathfrak{g})^* := \Hom_ \mathbb{Q}(U(\mathfrak{g}),\mathbb{Q})$ such that the evaluation pairing $(f,x) \to f(x)$ for $f\in  \mathbb{Q}[G]$ and $x\in U(\mathfrak{g})$ is non-degenerate. This turns $\mathbb{Q}[G]$ into a $U(\mathfrak{g})\otimes U(\mathfrak{g})$-module in the natural way. For $x\in \mathfrak{n}\oplus \mathfrak{n}_-$, 
\[
  ((x,\mathbbm{1})f)(g):=\frac{d}{dt}\Big|_{t=0}f\left(\exp(tx^T) g\right),\quad ((\mathbbm{1}, x)f)(g):=\frac{d}{dt}\Big|_{t=0}f\left(g\exp(tx) \right), 
\]
where $\mathbbm{1}$ is the unite in the algebra $U(\mathfrak{g})$. Note that both actions of $U(\mathfrak{g})$ are left actions. To distinguish them, let us denote by $U^L(\mathfrak{g})$ (resp. $U^R(\mathfrak{g})$) for the action of $U(\mathfrak{g})\otimes \mathbbm{1}$ (resp. $\mathbbm{1}\otimes U(\mathfrak{g})$). By algebraic Peter-Weyl Theorem, we have the following $U(\mathfrak{g})\otimes U(\mathfrak{g})$-modules isomorphism
\[
  \mathbb{Q}[G]\cong \bigoplus_{\lambda\in X^*_+(H)}V_{\lambda}\otimes V'_{\lambda},
\]
where $V_{\lambda}$ (resp. $V'_{\lambda}$) is the irreducible $U^L(\mathfrak{g})$ (resp. $U^R(\mathfrak{g})$) module with highest weight $\lambda$. Here as function on $G$, an element $v\otimes v'\in V_\lambda\otimes V'_\lambda$ evaluate at $g\in G$ by
\[
    v\otimes v'(g)=\langle v, g.v' \rangle,
\]
where $g.v'$ is the action of $g$ on $v'$, and $\langle \cdot,\cdot\rangle$ is unique paring such that $\langle v_\lambda,v_\lambda\rangle=1$ and $\langle v, g.v' \rangle=\langle (g^T).v, v' \rangle$. Again by algebraic Peter-Weyl Theorem, one finds
\begin{equation}\label{G/U}
  \mathbb{Q}[G]^U\cong \bigoplus_{\lambda\in X^*_+(H)}V_{\lambda}\otimes v'_\lambda. 
\end{equation}
Note that $\mathbb{Q}[U^-]$ is a $U(\mathfrak{n}_-)\otimes U(\mathfrak{n}_-)$-module. For each $f\in \mathbb{Q}[G]$, denote by $\pi^-\colon \mathbb{Q}[G]\to \mathbb{Q}[U^-]$ the restriction $f|_{U^-}$ of $f$ to $U^-$. The following is immediate:
\begin{Lem}\label{PRV}
  For any dominant integral weight $\lambda$, one has:
  \begin{itemize}
    \item The restriction of $\pi^-$ to $V_\lambda\otimes v'_\lambda$ is an injective homomorphism of $U^{L}(\mathfrak{g})$-module;

    \item The image of $V_\lambda\otimes v'_\lambda$ is described by
      \[
        \pi^-(V_\lambda\otimes v'_\lambda)=\left\{f_-\in \mathbb{Q}[U^-]  \ \Big|\  \left(e_i^*\right)^{\langle\lambda,\alpha_i^\vee \rangle +1}  f_-=0, \forall i\in I\right\},
      \]
      where $(e_i^*f)(u_-):=\dfrac{d}{dt}\Big|_{t=0}f\left(u_-\exp(tf_i)\right)$.
  \end{itemize}
\end{Lem}

As a corollary, for any function $f\in \mathbb{Q}[U^-]$, there exists a dominant integral weight $\lambda$ such that
  \begin{equation}\label{lambda}
    f\in \pi^-(V_{\lambda}\otimes v'_\lambda)\subset \mathbb{Q}[U^-].
  \end{equation}
Then by Lemma \ref{PRV}, there exists $F_\lambda\in \mathbb{Q}[G]^U$ such that $F_\lambda|_{U^-}=f$. The function $F_{\lambda}$ will be called a lift of $f$. Note that the lift of $f$ is not unique. Among the set $\Lambda(f)$ of $\lambda$'s satisfying \eqref{lambda} for $f\in \mathbb{Q}[U^-]$, there exists a minimal one $\lambda^{\min}$, in the sense that $\langle\lambda^{\min}, \alpha_i^\vee \rangle \leqslant \langle\lambda, \alpha_i^\vee \rangle$ for all $i\in I$ and all $\lambda\in \Lambda(f)$. We refer to $F_{\lambda^{\min}}$ the {\em minimal lift} of $f$. In general, it is not easy to find $\lambda^{\min}$. While, if $f$ is of the form $\pi^-(v_{w_0\mu}\otimes v)$, we have
\begin{Pro}\label{lift1}
  For an integral dominant weight $\mu$ and an element $v_{w_0\mu}\otimes v\in V_\mu\otimes V'_{\mu}$, where $v_{w_0\mu}$ is the lowest weight vector in $V_\mu$ and $v$ is any vector in $V'_\mu$. The weight $\lambda^{\min}$ of the minimal lift for $\pi^-(v_{w_0\mu}\otimes v)$ is given by $\lambda^{\min}=\sum c_i\omega_i$, where $c_i$ is smallest number such that $f_i^{c_i+1}v= 0$.
\end{Pro}
\begin{proof}
 We need to find the minimal $c_i$ such that
  \[
    \left(\left(e_i^*\right)^{c_i+1}\pi^-(v_{w_0\mu}\otimes v)\right)(u_-)= 0,\quad \forall u_-\in U_-,
  \]
  which is equivalent to find minimal $c_i$ such that
  \begin{equation}\label{mini}
    \langle v_{w_0\mu}, u_- f_i^{c_i+1} v \rangle=0, \forall u_-\in U_-.  
  \end{equation}
  Since $v_{w_0\omega_i}$ is the lowest weight vector, $(u_-)^Tv_{w_0\mu}$ runs over the whole module $V_{\mu}$. Then the condition \eqref{mini} boils down to find minimal $c_i$ such that $f_i^{c_i+1}v=0$, which is what we need to show here.  
\end{proof}

Next, let us introduce the so called generalized minors on $G$. For a dominant weight $\mu\in X^*_+(H)$ of $G$, the \emph{principal minor} $\Delta_\mu$ is a regular function $G\to \mathbb{G}_{\bf a}$ uniquely determined by
\[
  \Delta_\mu(u_-au):=\mu(a), \text{~for any~} u_-\in U^-, a\in H, u\in U.
\]
For any two weights $\gamma$ and $\delta$ of the form $\gamma=w\mu$ and $\delta=v\mu$, where $w,v\in W$, the {\em generalized minor} $\Delta_{w\mu,v\mu}$ is a regular function on $G$ given by
\[
  \Delta_{\gamma,\delta}(g)=\Delta_{w\mu,v\mu}(g):=\Delta_\mu(\overline w^{\,-1}g\overline v), \text{~for all~} g\in G.
\]
If $G={\rm SL}_n$, the generalized minors are minors.  Define the  {\em fundamental coweights} $\omega_i^\vee$ as the dual of simple roots: $\omega_i^\vee(\alpha_j)=\delta_{ij}$, and the {\em fundamental weights} $\omega_i$ as the dual of simple coroots: $\omega_i(\alpha_j^\vee)=\delta_{ij}$. Note that the fundamental coweights of $\mathfrak{g}$ are fundamental weights of the Langlands dual Lie algebra $\mathfrak{g}^\vee$ of $\mathfrak{g}$. Using generalized minors, the function $\Phi_{BK}$ can be written as\cite[Corollary 1.25]{BKII}: 
\begin{equation}\label{BKp}
  \Phi_{BK}=\sum_{i\in I}\frac{\Delta_{w_0\omega_i,s_i\omega_i}+\Delta_{w_0s_i\omega_i,\omega_i}}{\Delta_{w_0\omega_i,\omega_i}}\in \mathbb{Q}[G^{w_0}]. 
\end{equation}

For a dominant integral weight $\lambda \in X^*_+(H)$, which can be written uniquely as
\[
  \lambda = \sum_{i\in I} c_i(\lambda) \omega_i, \qquad c_i(\lambda)\in \mathbb{Z}_{\geqslant 0},
\]
the principal minor $\Delta_{\lambda}$ is of the form
\begin{equation}\label{deltadef}
  \Delta_{\lambda} = \prod_{i\in I} \Delta_{\omega_i}^{c_i(\lambda)}.
\end{equation}
One can check that for  $h,h'\in H$ and $e_i$ in the Chevalley generators
\[
  (h, h')\cdot \Delta_{\lambda}= h^{\lambda}{h'}^{\lambda} \Delta_{\lambda},  \qquad (e_i, \mathbbm{1})\cdot \Delta_{\lambda}= (\mathbbm{1}, e_i)\cdot\Delta_{\lambda}=0 \text{~for~} i\in I,
\]
thus we conclude that $\Delta_{\lambda}$ is the highest weight vector in the $U(\mathfrak{g})\otimes U(\mathfrak{g})$-module $V_{\lambda}\otimes V'_{\lambda}\subset \mathbb{Q}[G]$. 

Next we will use Proposition \ref{lift1} to find the weight of the minimal lift of $\Delta_{w_0\omega_i,s_i\omega_i}$, which is not $U$-invariant in $\Phi_{BK}$:
\begin{Pro}\label{lift}
  Denote by $F_{\lambda_i}$ the minimal lift of $\Delta_{w_0\omega_i, s_i\omega_i}\big|_{U^-}$. Then $\lambda_i=\omega_i+s_i\omega_i$.
\end{Pro}
\begin{proof}
  Note that $\Delta_{w_0\omega_i, s_i\omega_i}=v_{w_0\omega_i}\otimes v'_{s_i\omega_i}$.  Fix the index $i$, what we need to find is the minimal $c_j$ as in Proposition \ref{lift1}.
  Recall that for any integral dominant weight $\mu$ and $w\in W$, the following map
  \begin{equation}\label{extr}
    f_{j}^{\langle w\mu, \alpha_j^\vee\rangle}\colon V_{\mu}(w\mu)\to V_{\mu}(s_jw\mu)
  \end{equation}
  is an isomorphism if $\ell(s_jw)=\ell(w)+1$. Thus $v'_{s_i\omega_i}$ is a multiple of $f_iv'_{\omega_i}$, then we need to find minimal $c_j$ such that
  \[
    f_j^{c_j+1}f_iv'_{\omega_i}=0.
  \]
  By \eqref{extr}, one find $c_j=-a_{ji}$ for $j\neq i$ and $c_i=0$.  Thus we get $\lambda_i=-\sum_{j\neq i}a_{ji}\omega_j=\omega_i+s_i\omega_i$.  
\end{proof}

\begin{Cor}
  The function $F_{\lambda_i}$ is bi-homogeneous of degree $(w_0\omega_i+\omega_i,s_i\omega_i+\omega_i)$.
\end{Cor}
\begin{proof}
  By definition we have $\Delta_{w_0\omega_i,s_i\omega_i}|_{U^-}=F_{\lambda_i}|_{U^-}$. Since $F_{\lambda_i}$ is $U$-invariant, we get
  \[
    \Delta_{w_0\omega_i,s_i\omega_i}([g]_-)=F_{\lambda_i}(g)\Delta_{\lambda_i}^{-1}(g).
  \]
  By the uniqueness of Gaussian decomposition, one has $[hg]_-=h[g]_-h^{-1}$ and $[gh]_-=[g]_-$. Thus $F_{\lambda_i}$ is of bi-degree $(w_0\omega_i-s_i\omega_i+\lambda_i,\lambda_i)=(w_0\omega_i+\omega_i,s_i\omega_i+\omega_i)$.
\end{proof}

\begin{Ex}
  For $G=\SL_n$, we have $\lambda_i=\omega_{i-1}+\omega_{i+1}$. One can find the minimal lift of $\Delta_{w_0\omega_i,s_i\omega_i}$ by expanding each minor along the last column. For our convenience, we write them down for $n=4$
  \begin{align*}
    \Delta_{w_0\omega_1,s_1\omega_1}([g]_-)&=\frac{\Delta_{s_3s_2\omega_2,\omega_2}(g)}{\Delta_{\omega_2}(g)};\\
    \Delta_{w_0\omega_2,s_2\omega_2}([g]_-)&=\frac{\Delta_{s_2s_1\omega_1,\omega_1}(g)\Delta_{s_3\omega_3,\omega_3}(g)-\Delta_{w_0\omega_1,\omega_1}(g)\Delta_{\omega_3}(g)}{\Delta_{\omega_1+\omega_3}(g)};\\
    \Delta_{w_0\omega_3,s_3\omega_3}([g]_-)&=\frac{\Delta_{s_1s_2\omega_2,\omega_2}(g)}{\Delta_{\omega_2}(g)}.
  \end{align*}
\end{Ex}

\subsection{Combinatoric expressions of tensor multiplicities}
First of all, let us introduce the notion of $(\mathbf{i},K,L)$-trail, which is more general than the $\mathbf{i}$-trail introduced in \cite{BZ01}. For any finite dimensional representation $V$ of $\mathfrak{g}$, let $V=\oplus_{\gamma\in P}V(\gamma)$ be the weight decomposition. Let $P(V)$ be the set of weights of $V$, which is the set of elements $\gamma\in P$ such that $V(\gamma)\neq 0$.

Let $\gamma$ and $\delta$ be two weights (not necessary extremal) in $P(V)$. Given a decorated word $(\mathbf{i},K,L)$, let $ \mathbb{I}(\mathbf{i},K,L)=(j_1,\ldots,j_n)$ be the associated double word as before. An $(\mathbf{i},K,L)$-trail from $\gamma$ to $\delta$ in $V$ is a sequence of weights $\pi=(\gamma=\gamma_0,\gamma_1,\dots,\gamma_n=\delta)$ such that: 
\begin{itemize}
  \item[$(1)$] for $l= 1,\dots, n$, we have $\gamma_{l-1}-\gamma_l=\sign\left(j_l\right)c_l\alpha_{j_l}$ for some non-negative integer $c_l$;
  \item[$(2)$] $e^{(c_1)}_{j_1}\cdots e^{(c_n)}_{j_n}$ is a non-zero linear map from $V(\delta)$ to $V(\gamma)$,
\end{itemize}
where we use the (decorated) divided power
\[
  e_{i}^{(n)}:=
  \begin{dcases}
    e_{i}^{n}/n!, & \text{~if~} i>0\\
    f_{i}^{n}/n!, & \text{~if~} i<0
  \end{dcases}.
\]
Moreover, for a specific vector $f\in V(\delta)$, a $(\mathbf{i},K,L)$-trail {\em for $f$} is a $(\mathbf{i},K,L)$-trail satisfying
\begin{itemize}
  \item[$(2')$] $e^{(c_1)}_{j_1}\cdots e^{(c_n)}_{j_n}f$ is non-zero.
\end{itemize}
If $V$ is irreducible and the highest weights  $\gamma$ and $\delta$ are extremal, the condition $(2')$ is redundant since $\dim(V(\gamma))=1$. For a $(\mathbf{i},K,L)$-trail $\pi$ in $V_{\omega_i}$ from $\gamma$ to $\delta$, denote by $a_l$
\[
  a_l(\pi)=\left\{
  \begin{aligned}
    \dfrac{1}{2}(\gamma_{l-1}-\gamma_l)\alpha_{i_l}^\vee, &\quad \text{if~} j_l>0\\
    -\dfrac{1}{2}(\gamma_{l-1}+\gamma_l)\alpha_{i_l}^\vee, &\quad \text{if~} j_l<0
  \end{aligned}\right.
\]
\begin{Rmk}
  For the decoration $K=L=\emptyset$, the definition of $(\mathbf{i}, K,L)$-trail recovers the $\mathbf{i}$-trail in \cite[Definition 2.1]{BZ01} and we have $a_l=c_l$ which recovers \cite[Eq 2.1]{BZ01}.
\end{Rmk}

Recall that $\Delta_{w_0\omega_i, s_i\omega_i}$ has  minimal lift $F_{\lambda_i}\in V_{\lambda_i}$ with $\lambda_i=\omega_i+s_i\omega_i$, and $F_{\lambda_i}$ is bi-homogeneous of degree $(w_0\omega_i+\omega_i,\lambda_i)$. Denote by
\[
	f_{\lambda_i}:=F_{\lambda_i}\circ T.
\]
Recall that $v_K:=w_{\mathbf{i}_K}$ for a given decorated reduced word $(\mathbf{i} = (i_1, \dots, i_m),K,L)$.
\begin{Thm}\label{comb}
  Let $\lambda, \mu, \nu$ be three dominant weights for $\mathfrak{g}$ and $(\mathbf{i} = (i_1, \dots, i_m),K,L)$ be any decorated reduced word of $w_0$. Then the multiplicity $c_{\lambda, \nu}^{\mu}$ is equal to the number of integer $m$-tuples $(t_1, \dots, t_m)\in \mathbb{Z}^m$ satisfying the following conditions:
  \begin{itemize}
    \item $\sum t_lr_l\alpha_{i_l}=\lambda+\nu-\mu$, where $r_l=\prod_{k> l, j_k>0} s_{j_k}$ with product in decreasing order;

    \item $\sum a_l(\pi)t_l \geqslant 0$ for any $i$ and any $(\mathbf{i},K,L)$-trail $\pi$ from $v_k\omega_i^\vee$ to $s_i\omega_i^\vee$ in $V_{\omega_i^\vee}$;

    \item $\sum a_l(\pi)t_l \geqslant \langle \omega_i^\vee, \lambda+s_i\nu-\mu \rangle$ for any $i$ and any $(\mathbf{i},K,L)$-trail $\pi$ from $v_k\omega_i^\vee$ to $w_0s_i\omega_i^\vee$ in $V_{\omega_i^\vee}$;

    \item $\sum a_l(\pi)t_l \geqslant \langle \omega_i^\vee, s_i\lambda+\nu-\mu \rangle$ for any $i$ and any $(\mathbf{i},K,L)$-trail $\pi$ for $f_{\lambda_i^\vee}$ from $v_k \lambda_i^\vee$ to $w_0\omega_i^\vee+\omega_i^\vee$ in $V_{\lambda_i^\vee}$.
  \end{itemize}
\end{Thm}

\begin{Rmk}
  Note for $K=L=\emptyset$, it recovers \cite[Theorem 2.3]{BZ01}.
\end{Rmk}

\section{Geometric Lift of Reduction Multiplicities}\label{reduction multi}

Here we will consider reduction multiplicities following \cite[Section 2.4]{BZ01}. For a subset $J\subset I=\{1,\ldots,r\}$, denote by $\mathfrak{g}_J$ the corresponding {\em Levi subalgebra} generated by the Cartan subalgebra $\mathfrak{h}$ and $\{e_j,f_j\mid j\in J\}$. A weight $\beta\in P(\mathfrak{g})$ is {\em dominant for $\mathfrak{g}_J$} if $\langle \beta, \alpha_j^\vee\rangle\geqslant 0$ for all $j\in J$. Denote by $P_J(\mathfrak{g})$ the set of dominant weights for $\mathfrak{g}_J$. Denote by $V_\beta^J$ the irreducible $\mathfrak{g}_J$- module with highest weight $\beta$. What we want to compute now is the multiplicity of $V_\beta^J$ in the irreducible $\mathfrak{g}$-module $V_\lambda$. Denote by $w_0^J$ the longest element of the parabolic subgroup in $W$ generated by $s_j$ for any $j\in J$. Denote the Levi subgroup by $L\subset G$ and $w_P=(w_0^J)^{-1}w_0$.

\begin{Thm}\cite[Theorem 2.8]{BZ01}\label{reduc}
  For any reduced word $\mathbf{i} = i_1, \dots, i_n$ of $w_P=(w_0^J)^{-1}w_0$, the multiplicity of $V_\beta^J$ in the irreducible $\mathfrak{g}$-module $V_\lambda$ is equal to the number of $n$-tuples $(t_1, \dots, t_n)\in \mathbb{Z}^m$ satisfying the following conditions:
  \begin{itemize}
    \item $\sum_k d_k(\pi)t_k\geqslant 0$ for any $\mathbf{i}$-trail $\pi$ from $w_0^J\omega_i^\vee$ to $w_0s_i\omega_i^\vee$ in $V_{\omega_i^\vee}$;

    \item $\sum t_k\alpha_{i_k}=\lambda-\beta$;

    \item $t_k+\sum_{l>k} a_{i_k,i_l}t_l \leqslant \lambda(\alpha_{i_k}^\vee)$.
  \end{itemize}
\end{Thm}
\begin{Thm}
	The geometric lift of the reduction multiplicity is the following object:
	\[
		{\bm L}:=(L^{\vee;w_P,e}\times H^\vee, \Phi_L, \hw_L, \pi_L)\in \mathbf{Mult}_L
	\]
	where each component is given by
	\begin{align*}
		\Phi_L(x,h)&:=\sum_{i\in I} \Delta_{w_P\omega_i^\vee,s_i\omega_i^\vee}(x)+\sum_{i\in I} h^{-\alpha_i^\vee}\frac{\Delta_{s_iw_i^\vee,\omega_i^\vee}(x)}{\Delta_{w_i^\vee,\omega_i^\vee}(x)};\\
		\hw_L&\colon L^{\vee;w_P,e}\times H^\vee\to H^\vee \ :\ (x,h)\to h;\\
		\pi_L&\colon L^{\vee;w_P,e}\times H^\vee\to H^\vee \ :\ (x,h)\to h\cdot [x]_0.
	\end{align*}
\end{Thm}
\begin{proof}
	One only need to verify that for the following positive chart
	\[
		\eta\colon \mathbb{T}^n\to L^{\vee; w_P,e} \ :\ (t_1,\ldots,t_n)\mapsto x_{-i_1}^\vee(t_1)\cdots x_{-i_n}^\vee(t_n),
	\]
	the tropical fiber $(\hw_L,\pi_L)^{-t}(\beta,\lambda)$ gives the combinatoric expression in Theorem \ref{reduc}.
\end{proof}

As in the story of tensor multiplicities, different toric charts of $L^{\vee;w_P,e}\times H^\vee$ will bring us many more combinatoric expression for reduction multiplicities. Here we will only show the following example and leave the general case in the future work.

\begin{Ex}
  Let $G$ be $\SL_{n+1}$ and $J=\{1,\ldots,n\}\slash \{n-1\}$. Then the Levi subgroup $L$ is the subgroup of $\GL_{n-1}\times \GL_2$ consisting of elements $(x,y)$ satisfying $\det(x)\det(y)=1$, and $w_P=s_{n-1}\cdots s_1\cdot s_n\cdots s_2$. In what follows, we will give some new combinatoric expressions for the reduction multiplicities using the separable decorated word for $w_P$. Recall that by Theorem \ref{open}, we have the following open embedding
  \[
  	L^{w^{-1},v^{-1}}\hookrightarrow L^{w^{-1}v,e}\ :\ x\to [x\overline{v}]_-[x\overline{v}]_0
  \]
  where $v,w\in W$ satisfying $\ell(v)+\ell(w)=\ell(v^{-1}w)$. For $k\in[0,n-1]$, one can write $x\in L^{w_P,e}$ as
  \[
    x=[x_{-n+1}(a_{n-1})\cdots x_{-k}(a_k)\cdot x_{2}(b_2)\cdots x_{n}(b_n)\cdot x_{1}(a_1)\cdots x_{k-1}(a_{k-1}) \overline{s_{k-1}}\cdots \overline{s_1}\cdot \overline{s_n}\cdots \overline{s_1}]_{\leqslant 0}.
  \]
  Note that $x_i(t)\overline{s_i}=x_{-i}(t^{-1})x_i(-t^{-1})$. By \cite[Proposition 7.2]{BZ01}, one can rewrite $x$ as
  \[
    x=x_{-n+1}(p_{n-1})\cdots x_{-1}(p_1) \cdot x_{-n}(q_n)\cdots x_{-2}(q_2),
  \]
  where $p_i$'s and $q_j$'s are given by rational functions of $a_i$'s and $b_j$'s. To illustrate,  set $n=4$, and $k=3$,
  \begin{align*}
  	p_3&=a_3;\quad &p_2&=(a_2+b_2)^{-1}; \quad &p_1&=a_1^{-1}a_2^{-1};\\
  	q_4&=b_4^{-1};\quad &q_3&=a_2^{-1}b_3^{-1}b_4^{-1};\quad &q_2&=a_1^{-1}a_2^{-1}b_2^{-1}b_3^{-1}b_4^{-1}(a_2+b_2).
  \end{align*}
  Note that the reduction multiplicity in Theorem \ref{reduc} is the tropical fiber of ${\bm L}$ in coordinate $p_i$'s and $q_j$'s. The piecewise linear map between $(p_1,\ldots,p_{n-1};q_2,\ldots,q_n)^t$ and $(a_1,\ldots,a_{n-1};b_2,\ldots,b_n)^t$ will bring us new combinatoric expressions other than the one in Theorem $\ref{reduc}$. 

  For the word $w_P=(s_{n-1}\cdots s_1)\cdot (s_{n}\cdots s_2)$, using the tricks in Lemma \ref{App:Lemma} , one can get at least $n-2$ non-equivalent decorations, which will give us at least $n-2$ new combinatoric expressions (up to linear transformation).
\end{Ex}

\section{Proofs of Results in Section \ref{Geometric Multiplicities}}\label{section proofs1}

Note that we have the following two procedures: On the one hand, to a trivializable $(U\times U,\chi^{\st})$-bicrystals $(X,{\bm p},\Phi_X)$ with an $U\times U$-invariant function $\pi_X\colon X\to S$,  one has a geometric multiplicities as $(U\backslash X/U,\overline{\Delta}_X,\overline{\hw}_X, \overline{\pi}_X)$, where $\hw_X:=\hw\circ {\bm p}$.

On the other hand, using the fact that $G$ is a $(U\times U,\chi^{\st})$-bicrystal, for any ${\bm M}\in \mul_G$, one can construct a $(U\times U, \chi^{\st})$-bicrystal as follows:
\[
  X_M:=M\times_H G,\quad \alpha_M\colon U\times X_M \times U \to X_M\ : \ u\cdot (m,g)\cdot u' \mapsto (m,ugu'), 
\]
where the fiber product is over $\hw_M\colon M\to H$ and $\hw\colon G\to H$, and moreover
\[
  {\bm p}_M\colon X_M\to G\ : \ (m,g)\mapsto g, \quad \Phi_{X_M}(m,g):=\Phi_M(m)+\Phi_{BK}(g).
\]

\begin{proof}[Proof of Theorem \ref{Associator}]
  We will show first that the equivalence of category $\mathbf{TriUB}_G$ and $\mul_G$. From the discussion above,  we have the following functors:
  \begin{align*}
    \mathcal{F}\colon \mathbf{TriUB}_G&\to \mul_G\ :\ (X\cong U\backslash X/U\times_H G, \Phi_X) \to (U\backslash X/U, \overline{\Delta}_X);\\
    \mathcal{G}\colon \mul_G&\to \mathbf{TriUB}_G\ :\ (M, \Phi_M) \to (M\times_H G, \Phi_M+\Phi_{BK}).
  \end{align*}
  One can check that $\mathcal{F}\mathcal{G}$ (resp. $\mathcal{G}\mathcal{F}$) is natural isomorphic to the identity functor for $\mul_G$ (resp. $\mathbf{TriUB}_G$). Moreover, by definition we have $\mathcal{G}(M\star N)\cong \mathcal{G}(M)*\mathcal{G}(N)$ and the following commuting diagram:
  \begin{equation}
  \begin{tikzcd}
    \mathcal{G}((M_1\star M_2)\star M_3) \arrow[r, " \mathcal{G}(\widetilde{\Psi}_{M_1,M_2,M_3})"] \arrow[d, "\sim"'] &[1.3em] \mathcal{G}(M_1\star (M_2\star M_3)) \arrow[d, "\sim"] \\[1em]
    (\mathcal{G}(M_1)*\mathcal{G}(M_2))*\mathcal{G}(M_3) \arrow[r, "\sim"] &  \mathcal{G}(M_1)*(\mathcal{G}(M_2)*\mathcal{G}(M_3))
  \end{tikzcd}\ ,
  \end{equation}
  where all the $\sim$'s are natural isomorphisms. Since $\mathbf{TriUB}_G$ is a monoidal category with trivial associator, we know $\mul_G$ is a monoidal category with associator given by the formula \eqref{assoformula}. 
\end{proof}

Following the spirit of \cite{BKII}, we say $(U\times U, \chi^{\st})$-bicrystal $({\bm X},{\bm p}, \Phi)$ is {\em positive trivializable} if there exist positive structures for $X^-:={\bm p}^{-1}(B^-)$ and $U\backslash X/U$ respectively, such that the map $\varphi$ in \eqref{trivializable}  and its inverse $\varphi^{-1}$ restrict to positive birational isomorphisms of unipotent bicrystals:
\begin{equation}
  \varphi_-\colon X^-\xleftrightarrow{\sim} U\backslash X/U\times_H B^- \ \mathpunct{:}(\varphi^{-1})_-.
\end{equation}

\begin{proof}[Proof of Theorem \ref{positiveasso}]
  Because of Theorem \ref{Associator} and Remark \ref{rmkofposass}, what we need to show is that the following map $F$ and its inverse $F^{-1}$ are isomorphisms of positive varieties with potential:
  \begin{equation}\label{mainmap}
    F \colon B^-\times B^- \to (U\times H^2)\times_H B^- \ :\ (g_1,g_2)\mapsto \left( \pi(g_1,g_2), \hw(g_1), \hw(g_2), g_1g_2\right).
  \end{equation}
  By \cite[Claim 3.41]{BKII}, we know that $F$ is positive. What left is to show $F^{-1}$ is a positive isomorphism.

  In what follows, we will first give explicit formulas \eqref{v2}-\eqref{v1u2} for the inverse of $F$. Then show that \eqref{v2}-\eqref{v1u2} are positive with respect to the positive structures given by Lemma \ref{positiveonU}.

  Note that $L^{e,w_0}$ is open dense in $U$ and $G^{w_0,e}$ is open dense in $B^-$. Let $(u,h_1,h_2,y)\in (L^{e,w_0}\times H^2)\times_H G^{w_0,e}$. We will find the expression of $g_i$ in terms of $u,y$ and $h_i$. Since
  \[
    \overline{\hw}_{M^{(2)}}(u,h_1,h_2)=\hw(y),
  \]
  there exists a unique pair $(u_1, v_2)\in U\times U$ such that $y=u_1\overline{w_0}h_1^{w_0}uh_2\overline{w_0}v_2$, where $h^{w_0}$ is short for $\overline{w_0}^{-1}h\overline{w_0}$. Denote by $x=h_1^{w_0}uh_2$ for simplicity. Then by taking $[\cdot]_+$ part of $\overline{w_0}^{-1}y$, we have
  \[
    [\overline{w_0}^{-1}y]_+=[\overline{w_0}^{-1}u_1\overline{w_0}x\overline{w_0}v_2]_+=[x\overline{w_0}]_+v_2
  \]
  since $\overline{w_0}^{-1}u_1\overline{w_0}\in B_-$ and $v_2\in U$. Therefore we have
  \begin{equation}\label{v2}
    v_2=[\overline{w_0}^{-1}y]_+[x\overline{w_0}]_+^{-1}.
  \end{equation}
  Now let's applying $\iota$ to $y$, we have $y^{\iota}=v_2^\iota \overline{w_0}x^{\iota}\overline{w_0}u_1^\iota$. Then using \eqref{v2}, we get
  \begin{equation}\label{u1}
    u_1^{\iota}=[\overline{w_0}^{-1}y^\iota]_+[x^\iota\overline{w_0}]_+^{-1}.
  \end{equation}

  In order to write $g_i$ as $u_ih_i\overline{w_0}v_i$, we just need to define
  \begin{gather}
  \begin{aligned}\label{v1u2}
    v_1=[u_1h_1\overline{w_0}]_+^{-1}, &\quad \text{~and~} \quad u_2^\iota=[v_2^\iota \overline{w_0}h_2^{-1}]_+^{-1}.
  \end{aligned}
  \end{gather}
  Now one can easily check that \eqref{v1u2} does give the inverse of $F$.

	What's next is to show the positivity. Note that the restriction of $\iota$ on $G^{w_0,e}$ is positive with respect to the positive structure. Thus the positivity of \eqref{v2} implies the positivity of \eqref{u1}. What we will show next is that the two factors of \eqref{v2} 
	\[
		\eta\colon G^{w_0,e}\to U \ :\ g\mapsto [\overline{w_0}^{-1}g]_+\quad \text{~and~}\quad \zeta\colon G^{e,w_0} \to U\ :\ g\mapsto [g\overline{w_0}]_+^{-1}
	\]
	are positive. Then the positivity of \eqref{v1u2} follows from the positivity of the map $\zeta$. 

	First, by \cite[Claim 3.25,(3.6)]{BKII} and the fact that $\iota$ is positive, one knows $\eta$ is positive.

	Second, write $b:=[g\overline{w_0}]_-[g\overline{w_0}]_0$, then one has $\zeta(g)=\overline{w_0}^{-1}g^{-1}b$. Applying $\left(\cdot \right)^\iota$ to both side of equation $\zeta(g)=\overline{w_0}^{-1}g^{-1}b$, we get
	\begin{equation}\label{zeta}
		\zeta(g)^\iota=b^\iota \overline{w_0}^{-1} \sigma(g),
	\end{equation}
	where $\sigma(g)=\overline{w_0}g^{-\iota}\overline{w_0}^{-1}$. Thus we get $\zeta(g)^\iota=[\overline{w_0}^{-1}\sigma(g)^T]_+$. Since $\sigma(g)^T$ is positive by \cite[Eq (4.6)]{BZ01}, the map $\zeta$ is positive.
\end{proof}

\section{Proofs of Results in Section \ref{tensor mul}}

In \cite{BKII}, the authors constructed a functor
\[
  \mathcal{B}\colon \mathbf{UB}_G^+\to \mathbf{Mod}_{G^\vee}
\]
from the category $\mathbf{UB}_G^+$ of positive  unipotent bicrystals \cite[Definition 3.29]{BKII} to the category $\mathbf{Mod}_{G^\vee}$ of $G^\vee$ module by passing through the geometric crystals and Kashiwara crystals \cite[Claim 6.9, 6.10, 6.12, Theorem 6.15]{BKII}. Here we briefly recall some properties of the functor $\mathcal{B}$.

Let $({\bm X},{\bm p},\Phi)$ be a positive unipotent bicrystals. Denote by $\hw_X:=\hw\circ {\bm p}\colon X\to H$ the highest weight map of $X$. In what follows, we write ${\bm X}$ for $({\bm X},{\bm p},\Phi)$ for simplicity. Then we have 
\begin{itemize}
  \item[$(1)$] $\mathcal{B}({\bm X}*{\bm X}')\cong \mathcal{B}({\bm X})\otimes \mathcal{B}({\bm X}')$ and $\mathcal{B}$ is monoidal.
\end{itemize}
Denote by $\pi_X\colon X\to S$ an $U\times U$-invariant positive map to a torus $S$. Then the $G^\vee$-module $\mathcal{B}({\bm X})$ can be parametrized over $\xi\in X_*(S)$ as direct sums of $G^\vee$-submodules, {\em i.e.},
\begin{itemize}
  \item[$(2)$] $\mathcal{B}({\bm X})=\bigoplus_{\xi\in X_*(S)} \mathcal{B}_\xi({\bm X}).$
\end{itemize}
Moreover, the typical components respect the convolution product $*$:
\begin{itemize}
  \item[$(1')$]  $\mathcal{B}_{\xi_1,\xi_2}({\bm X}_1*{\bm X}_2)\cong \mathcal{B}_{\xi_1}({\bm X}_1)\otimes\mathcal{B}_{\xi_2}({\bm X}_2)$.
\end{itemize}
For the positive unipotent bicrystal ${\bm G}:=(G,\id_G, \Phi_{BK})$ and the $U\times U$-invariant map $\hw\colon G\to H$:
\begin{itemize}
  \item[$(3)$] For $\lambda^\vee\in X_*^+(H) $, one has $\mathcal{B}_{\lambda^\vee}({\bm G})\cong V_{\lambda^\vee}$, where $V_{\lambda^\vee}$ is the irreducible $G^\vee$ module with highest weight $\lambda^\vee$; for $\lambda^\vee\notin X_*^+(H) $, one has $\mathcal{B}_{\lambda^\vee}({\bm G})=\emptyset$.
\end{itemize}
Let $(M,\Phi_M)$ be positive variety with potential (positive) fibered over torus $H\times S$. For unipotent bicrystal $({\bm X}=(X,\alpha), {\bm p}, \Phi)$, denote by ${\bm X}_M:=(M\times_H X, \alpha')$, where $\alpha'(u,(m,x),u)=(m, \alpha(u,x,u'))$. Thus $({\bm X}_M, {\bm p}, \Phi_M+\Phi)$ is a positive unipotent bicrystal (positive) fibered over $H\times S$, then we have for $(\lambda^\vee, \xi)\in X_*(H)\times X_*(S)$
\begin{itemize}
  \item[(4)]   $\mathcal{B}_{\lambda^\vee, \xi}({\bm X}_M)\cong \mathbb{C}[M^t_{\lambda^\vee, \xi}]\otimes \mathcal{B}_{\lambda^\vee}({\bm X})$, where $M^t_{\lambda^\vee, \xi}$ is the tropical fiber of $M$ over $(\lambda^\vee, \xi)$.
\end{itemize}

\begin{proof}[Proof of Theorem \ref{functor}]
  Note that we have the following functor
  \begin{align*}
    \mathcal{G}\colon \mul_G^+&\to \mathbf{TriUB}_G^+\ :\ (M, \Phi_M) \to (M\times_H G, \Phi_M+\Phi_{BK}).
  \end{align*}
  Then the assignment ${\bm M}\mapsto \mathcal{V}({\bm M})$ is just the combination of $\mathcal{B}\circ\mathcal{F}$ by the properties $(3,4)$ we list above. Thus we conclude that $\mathcal{V}$ is a monoidal functor.
\end{proof}

\begin{proof}[Proof of Theorem \ref{functor-compo}]
  First note that $\mathcal{V}_\xi({\bm M})\cong \mathcal{B}_{\xi}(\mathcal{F}({\bm M}))$. Then by the property $(1, 1', 4)$ listed above, we have:
  \begin{align*}
    \mathcal{V}_{\xi_1,\xi_2,\lambda^\vee, \nu^\vee}({\bm M}_1\star {\bm M}_2) &\cong \mathcal{B}_{\lambda^\vee, \xi_1}\left(\mathcal{F}({\bm M}_1)\right)\otimes \mathcal{B}_{\nu^\vee,\xi_2}\left(\mathcal{F}({\bm M}_2)\right)\\
    &\cong I_{\lambda^\vee}\left(\mathcal{V}_{\xi_1}({\bm M}_1)\right)\otimes I_{\nu^\vee}\left(\mathcal{V}_{\xi_1}({\bm M}_2)\right),
  \end{align*}
  which is what want to show.
\end{proof}

\section{Proofs of Results in Section \ref{combforten}}

In this section, we will show that how to use certain positive structure on geometric multiplicity $H\star H\cong M^{(2)}$ to give a combinatoric expression for the tensor multiplicities. In this section, to a decorated word $(\mathbf{i},K,L)$, let $ \mathbb{I}(\mathbf{i},K,L)=(j_1,\ldots,j_n)$ be the associated double word as before.

First, let us explain how to express functions $f\in \mathbb{Q}[G]$ by using the factorization parameters. Denote by
\[
  z_{i}:=
  \begin{dcases}
    x_{i}, & \text{~if~} i>0\\
    y_{i}, & \text{~if~} i<0
  \end{dcases}.
\]
The valuation of $f\in \mathbb{Q}[G]$ at $z_{\mathbb{I}}(t_1,\ldots,t_n)=z_{j_1}(t_1)\cdots z_{j_n}(t_n)$ defines a homomorphism of algebras:
\[
  \mathbb{Q}[G]\to \text{Sym}[t_1^\pm,\ldots, t_n^\pm]\ :\ f\mapsto f(z_{\mathbb{I}}(t_1,\ldots,t_n)).
\]
Consider $\mathbb{Q}[G]$ as $U(\mathfrak{g})\otimes U(\mathfrak{g})$ module as before. The coefficient of each monomial $t_1^{c_1}\cdots t_l^{c_n}$ in $f(z_{\mathbb{I}}(t_1,\ldots,t_n))$ is given by
\[
  \left( \left(\mathbbm{1}, e_{j_1}^{(c_1)}\cdots e_{j_n}^{(c_n)}\right)f\right)(e),
\]
where $e$ is the identity element in $G$ and $e_{i}^{(n)}$ is the (decorated) divided power
\[
  e_{i}^{(n)}:=
  \begin{dcases}
    e_{i}^{n}/n!, & \text{~if~} i>0\\
    f_{i}^{n}/n!, & \text{~if~} i<0
  \end{dcases}.
\]
This is simply by induction on the length of $\mathbb{I}$ and the following observation:
\[
  f\left(g\exp(te_i)\right)=\sum t^n\left(\left(\mathbbm{1},e_i^{(n)}\right)f\right)(g).
\]
If $f$ bi-homogeneous of degree $(\gamma,\delta)$, the value $f(e)$ can be non-zero only if $\gamma=\delta$. In this case, the evaluation $f(z_{\mathbb{I}}(t_1,\ldots,t_n))$ contains only monomial $t_1^{c_1}\cdots t_n^{c_n}$ such that $\sum \sign(j_l)c_{l}\alpha_{j_l}=\gamma-\delta$.

\begin{Lem}\label{trailtocomput}
  For a decorated word $(\mathbf{i},K,L)$ with associated double word $\mathbb{I}=(j_1,\ldots,j_n)$ and a dominant integral weight $\lambda$, a bi-homogeneous function $f_{w\lambda,\delta}\in V_\lambda\otimes V'_\lambda\subset \mathbb{Q}[G]$ of degree $(w\lambda,\delta)$ for some $w\in W$ has evaluation
  \[
    f_{w\lambda,\delta}(z_{\mathbb{I}}(t_1,\cdots,t_n))=\sum_{\pi} N_\pi t_1^{c_1(\pi)}\cdots t_n^{c_n(\pi)},
  \]
  where $N_\pi\in \mathbb{Q}^*$ and the sum on RHS is over all $(\mathbf{i},K,L)$-trails for $f_{w\lambda,\delta}$ from $w\lambda$ to $\delta$ in $V_\lambda$.
\end{Lem}
\begin{proof}
  Denote by $\gamma=w\lambda$. Note that for $v\in V_\lambda(\delta)$, 
  \[
    h(e_{j}^{(c)}v)=\left[h,e_{j}^{(c)}\right]v+e_{j}^{(c)}hv=\langle h, \delta+\sign(j)c\alpha_{j} \rangle e_{j}^{(c)}v.
  \]
  Thus for $c_l\geqslant 0$ satisfying $\sum \sign(i_l)c_{l}\alpha_{j_l}=\gamma-\delta$, the linear map
  \[
    \left(\mathbbm{1}, e_{j_1}^{(c_1)}\cdots e_{j_n}^{(c_n)}\right)\colon  V_\lambda(\gamma)\otimes V'_\lambda(\delta)\to V_\lambda(\gamma)\otimes V'_\lambda(\gamma)
  \]
  is well defined. Since $\gamma=w\lambda$ is extremal, we know the $V_\lambda(\gamma)\otimes V'_\lambda(\gamma)$ is one dimensional. Thus the homomorphism $\left(\mathbbm{1}, e_{j_1}^{(c_1)}\cdots e_{j_n}^{(c_n)}\right)$ is non-zero  if and only if 
  \[
    \left( \left(\mathbbm{1}, e_{j_1}^{(c_1)}\cdots e_{j_n}^{(c_n)}\right)f\right)(e),
  \]
  which is exactly the definition of $(\mathbf{i},K,L)$-trail $\pi$ for $f_{w\lambda,\delta}$ in $V_\lambda$. 
\end{proof}

\begin{Lem}\cite[Lemma 6.1]{BZ01}{\label{xtoz}}
  For any double word $(j_1,\dots,j_n)$, we have:
  \[
    x_{j_1}(t_1)\cdots x_{j_n}(t_n)=z_{j_1}(t_1')\cdots z_{j_n}(t_n')\cdot \prod_{j_l<0}t_l^{-\alpha_{j_l}^\vee}, \quad \text{~where~} t_l'=t_l\displaystyle\prod_{k<l, j_k<0}t_i^{\sign(-i_l)a_{j_k,j_l}}.
  \]
\end{Lem}
\begin{proof}
	First of all, write $x_{i_l}(t)=y_{j_l}(t)t^{-\alpha_{i_l}^\vee}$ for $j_l<0$. Also note we have: $hx_i(t)=x_i(h^{\alpha_i}t)h$ and $hy_i(t)=y_i(h^{-\alpha_i}t)h$ for any $h\in H$. 
\end{proof}

\begin{Pro}\label{tropoffun}
  For a decorated word $(\mathbf{i},K,L)$ with associated double word $\mathbb{I}=(j_1,\ldots,j_n)$ and a dominant integral weight $\lambda$, a bi-homogeneous function $f_{w\lambda,\delta}\in V_\lambda\otimes V'_\lambda\subset \mathbb{Q}[G]$ of degree $(w\lambda,\delta)$ for some $w\in W$ has evaluation
  \[
    f_{w\lambda,\delta}(x_{\mathbb{I}}(t_1,\cdots,t_n))=\sum_{\pi} N_\pi t_1^{a_1(\pi)}\cdots t_m^{a_m(\pi)},
  \]
  where $N_\pi\in \mathbb{Q}^*$ and the sum on RHS is over all $(\mathbf{i},K,L)$-trails for $f_{w\lambda,\delta}$ from $w\lambda$ to $\delta$ in $V_\lambda$, and
  \[
  	a_l(\pi)=\left\{
  	\begin{aligned}
    	\dfrac{1}{2}(\gamma_{l-1}-\gamma_l)\alpha_{i_l}^\vee, &\quad \text{if~} j_l>0\\
    	-\dfrac{1}{2}(\gamma_{l-1}+\gamma_l)\alpha_{i_l}^\vee, &\quad \text{if~} j_l<0
  	\end{aligned}\right.
  \]
  Moreover if $f_{w\lambda,\delta}$ is a positive rational function, then
  \[
    f_{w\lambda,\delta}^t(x_{\mathbf{i}}^\sigma(t_1,\cdots,t_m))=\min_{\pi}\left\{ {a_1(\pi)}t_1+\cdots+ {a_m(\pi)}t_m \right\},
  \]
  where the RHS is over all $(\mathbf{i},(k,\sigma))$-trails for $f_{w\lambda,\delta}$ from $w\lambda$ to $\delta$ in $V_\lambda$.
\end{Pro}
\begin{proof}
  The first part of the statement is just the combination of Lemma \ref{trailtocomput} and Lemma \ref{xtoz}. Note that $c_l=\sign(j_l)(\gamma_{l-1}-\gamma_l)\alpha_{j_l}^\vee/2$. Focus on monomial term. For $j_l>0$, the parameter $t_l$ only appears in $t_l'$ who has power $c_l$. So the power of $t_l$ is $c_l$. For $j_l<0$, the parameter $t_l$ appears in $t_k'$ for $k>l$ who has power $c_k$ and in the additional torus $\prod_{j_l<0}t_l^{-\alpha_{j_l}^\vee}$. Put all this together, we have
  \[
  	\langle-\alpha_{j_l}^\vee, \delta \rangle+\sum_{k>l}c_k\sign(-i_k)\langle\alpha_{j_k}, \alpha_{j_l}^\vee\rangle+c_l=-\dfrac{1}{2}(\gamma_{l-1}+\gamma_l)\alpha_{i_l}^\vee.
  \]
  For the second part, note that for trails $\pi$ and $\pi'$, the vector 
  \[
  	({a_1(\pi)}t_1,\ldots,{a_m(\pi)}t_m), \ \text{and}\ ({a_1(\pi')}t_1,\ldots,{a_m(\pi')}t_m)
  \]
  are linear independent. By \cite[Corollary 4.10(a)]{BKII}, we find the result.
\end{proof}
\begin{Rmk}
  For the simply-laced simple group, one can actually show that all the coefficients $N_\pi$ of $f_{\lambda_i}$ are positive. One may extend this result to any semi-simple group. What we used in the proof is weaker than all coefficients $N_\pi$ are positive. 
\end{Rmk}

By this proposition, one can describe the tropicalization of minimal lift $f_{\lambda_i}:=F_{\lambda_i}\circ T$ since we know that $f_{\lambda_i}$ is a positive rational function having bidegree $(\lambda_i,w_0\omega_i+\omega_i)$.

\begin{proof}[Proof of Theorem \ref{comb}]
By Lemma \ref{positiveonU}, to decorated reduced word $(\mathbf{i},K,L)\in \widehat{R}(w_0)$ we have a toric chart $\xi_{\mathbb{I}}$ on $L^{e,w_0}$. Then we need to compute the potential $\overline{\Delta}_2$ on $U\times H^2$, where
\[
   \overline{\Delta_{2}}(u,h_1,h_2)=\chi^{\st}(u)+\Phi_{BK}(h_2u^Th_1^{w_0}),
\]
  Recall that the function $\Phi_{BK}$ can be written using generalized minors as: 
\[
  \Phi_{BK}=\sum_{i\in I}\frac{\Delta_{w_0s_i\omega_i,\omega_i}+\Delta_{w_0\omega_i,s_i\omega_i}}{\Delta_{w_0\omega_i,\omega_i}}\in \mathbb{Q}[G^{w_0}], 
\]
and  the stand character $\chi^{\st}$ can be written as $\chi^{\st}=\sum \Delta_{\omega_i,s_i\omega_i}$. 

Next let's compute each term in the potential $\overline{\Delta}_2$ using $\xi_{\mathbb{I}}$. Denote by $v=v_K$ for simplicity. Then for any $w\in W$, we compute the characters of $u$ first:
\[
  \Delta_{\omega_i,s_i\omega_i}(u)=\Delta_{\omega_i, s_i\omega_i}([\overline{v}^{-1}x_{\mathbb{I}}]_+)=\frac{\Delta_{\omega_i, s_i\omega_i}([\overline{v}^{-1}x_{\mathbb{I}}]_{\geqslant 0})}{\Delta_{\omega_i, \omega_i}([\overline{v}^{-1}x_{\mathbb{I}}]_0)}=\frac{\Delta_{v\omega_i, s_i\omega_i}(x_{\mathbb{I}})}{\Delta_{v\omega_i, \omega_i}(x_{\mathbb{I}})}=\Delta_{v\omega_i, s_i\omega_i}(x_{\mathbb{I}}).
\]
where we use the fact $\Delta_{v\omega_i, \omega_i}(x)=1$ for $x\in L^{v,v'}$. Then we compute the minors of form $\Delta_{w\omega_i,\omega_i}$:
\begin{align*}
  \Delta_{w\omega_i,\omega_i}(h_2u^Th_1^{w_0})&=h_1^{w_0\omega_i}h_2^{w\omega_i}\Delta_{\omega_i, w\omega_i}([\overline{v}^{-1}x_{\mathbb{I}}]_+)=h_1^{w_0\omega_i}h_2^{w\omega_i}\frac{\Delta_{\omega_i, w\omega_i}([\overline{v}^{-1}x_{\mathbb{I}}]_{\geqslant 0})}{\Delta_{\omega_i, \omega_i}([\overline{v}^{-1}x_{\mathbb{I}}]_0)}\\
  &=h_1^{w_0\omega_i}h_2^{w\omega_i}\frac{\Delta_{v\omega_i, w\omega_i}(x_{\mathbb{I}})}{\Delta_{v\omega_i, \omega_i}(x_{\mathbb{I}})}=h_1^{w_0\omega_i}h_2^{w\omega_i}\Delta_{v\omega_i, w\omega_i}(x_{\mathbb{I}}).
\end{align*}
For $w=w_0$, we have:
\[
  \Delta_{v_K\omega_i,w_0\omega_i}(x_{\mathbb{I}})=\Delta_{w_K\omega_{i^*},\omega_{i^*}}\left(((x_{\mathbb{I}})^\iota)^T\right).
\]
Note that $(x_i^\iota(t))^T=x_{-i}(t)t^{\alpha_{i}^\vee}$. Moreover, we have the commutation relation: for $j\in I$ and $p\in \mathbb{G}_{\bf{m}}$,
\[
  p^{\lambda^\vee}x_{j}(t)=x_j(tp^{\langle \alpha_j,\lambda^\vee\rangle})p^{\lambda^\vee} \text{~and~} p^{\lambda^\vee}x_{-j}(t)=x_{-j}(tp^{-\langle \alpha_j,\lambda^\vee\rangle})p^{s_j\lambda^\vee}.
\]
Thus one can rewrite
\[
  ((x_{\mathbb{I}}(t_1,\ldots,t_m))^\iota)^T=x_{-\mathbb{I}}(t'_1,\ldots,t'_m)\cdot \prod t_l^{r_l \alpha_{i_l}^\vee},
\]
where $r_l=\prod_{k> l, j_k>0} s_{j_k}$ in decreasing order. Since $x_{-\mathbb{I}}(t'_1,\ldots,t'_m)\in L^{w_K,v_K}$, so we get
\[
  \Delta_{v_K\omega_i,\omega_i}(x_{\mathbb{I}}(t_1,\ldots,t_m))=\prod t_l^{\langle \omega_i,r_l\alpha_{i_l}^\vee\rangle}.
\]

At then end, the last term $\Delta_{w_0\omega_i,s_i\omega_i}$ can be computed using the minimal lift $F_{\lambda_i}$:
\begin{align*}
	\Delta_{w_0\omega_i,s_i\omega_i}(h_2u^Th_1^{w_0})&=h_1^{w_0s_i\omega_i}h_2^{w_0\omega_i}\Delta_{w_0\omega_i, s_i\omega_i}([\overline{v}^{-1}x_{\mathbb{I}}]_+^T)=h_1^{w_0s_i\omega_i}h_2^{w_0\omega_i}\frac{F_{\lambda_i}([\overline{v}^{-1}x_{\mathbb{I}}]_{\geqslant 0}^T)}{\Delta_{\lambda_i, \lambda_i}([\overline{v}^{-1}x_{\mathbb{I}}]_0)}\\
 	&=h_1^{w_0s_i\omega_i}h_2^{w_0\omega_i}\frac{F_{\lambda_i}\big((\overline{v}^{-1}x_{\mathbb{I}})^T\big)}{\Delta_{v\lambda_i, \lambda_i}(x_{\mathbb{I}}}=h_1^{w_0s_i\omega_i}h_2^{w_0\omega_i}F_{\lambda_i}\big((\overline{v}^{-1}x_{\mathbb{I}})^T\big),
\end{align*}
where we use the fact
\[
	\Delta_{\lambda_i, \lambda_i}(\overline{v}^{-1}x_{\mathbb{I}})=\Pi_{j\neq i} \Delta_{\omega_j, \omega_j}^{-a_{ji}}(\overline{v}^{-1}x_{\mathbb{I}})=1
\]
since $\lambda_i=\omega_i+s_i\omega_i=-\sum_{j\neq i}a_{ji}\omega_j$. Put all this together, we get
\begin{align}\label{potentialonM_2}
	\overline{\Delta_{2}}(u,h_1,h_2)=&\sum \Delta_{\omega_i,s_i\omega_i}(u)+\sum \frac{\Delta_{w_0s_i\omega_i,\omega_i}+\Delta_{w_0\omega_i,s_i\omega_i}}{\Delta_{w_0\omega_i,\omega_i}}(h_2u^Th_1^{w_0})\nonumber\\
	=&\sum \Delta_{v\omega_i, s_i\omega_i}(x_{\mathbb{I}})+\sum h_2^{\alpha_i^*}\frac{\Delta_{v\omega_i, w_0s_i\omega_i}(x_{\mathbb{I}})}{\Delta_{v\omega_i, w_0\omega_i}(x_{\mathbb{I}})}+\sum h_1^{\alpha_i^*}\frac{F_{\lambda_i}\big((\overline{v}^{-1}x_{\mathbb{I}})^T\big)}{\Delta_{v\omega_i, w_0\omega_i}(x_{\mathbb{I}})}.
\end{align}
Now by Proposition \ref{tropoffun}, the function \eqref{potentialonM_2} is Laurent polynomial in $t_i$'s. The coefficient $N_\pi$ for each function appeared in \eqref{potentialonM_2} is positive integral by \cite[Theorem 5.8]{BZ01} and Lemma \ref{positiveonU}. Thus the tropicalization of \eqref{potentialonM_2} gives a tropical cone in terms of $(\mathbf{i},(k,\sigma))$-trails.

Recall from definition of $H\star H\cong M^{(2)}$, we have the following projections:
\begin{align*}
    \hw_U=\hw_{H\star H}&\colon M^{(2)}\to H\ : \ (u,h_1,h_2)\mapsto h_1h_2\hw(\overline{w_0}u\overline{w_0}), \\
    \pi_{H\star H} &\colon M^{(2)}\to H^2\ : \ (u,h_1,h_2)\mapsto (h_1,h_2).
\end{align*} 
Note the highest weight map $\hw$ is determined by gennerlized minors in the following sense
\begin{equation}\label{hwproperty}
	\left(\hw(g)\right)^{w_0\omega_i}=\Delta_{w_0\omega_i,\omega_i}(g), \quad \forall g\in G^{w_0,e}, i\in \bm{I}.  
\end{equation}
Therefor we have: write $u=[\overline{v}^{-1}x_{\mathbb{I}}]_-^T$, 
\[
	\left( \hw_{H\star H}(u,h_1,h_2)\right)^{w_0\omega_i}=(h_1h_2)^{w_0\omega_i}\Delta_{v\omega_i, w_0\omega_i}(x_{\mathbb{I}}).
\]
Together with the tropical cone defined by \eqref{potentialonM_2}, one find the conclusion.
\end{proof}

\begin{appendices}
	\begin{proof}[Proof of Proposition \ref{inversemapapp}]
		First of all, since $\mathbb{Q}[\mathbb{A}^m]$ is an UFD, thus the rational function $\varphi_k$ is Laurent monomial in some irreducible polynomials in $m$ variables. Collect all these irreducible polynomials all $k\in [1,m]$, and denote them by $F_1,\ldots,F_{m'}$. 

		Next, we will show that composition of $\varphi$ with $F_i$ is a character of ${\mathbb T}_m$, {\em i.e.}, a monomial in $t_1,\ldots,t_m$. If not, suppose that $F_1$ does not satisfy this property, {\em i.e.}, $F(\varphi(t_1,\ldots,t_m)$ is not a monomial in $t_1,\ldots,t_m$. Thus, there exists a point $\tau=(\tau_1,\ldots,\tau_m)\in {\mathbb T}_m({\mathbb C})$ such that $F_1(\varphi(\tau))=0$ but $F_k(\varphi(\tau))\ne 0$ for $k>1$. This is a contradiction because $t_\ell=\prod_{k\geqslant 1} F_k^{a_k}$ for some $\ell$ with $a_1\ne 0$ and this fails for $t=\tau$. 

		Now we know, $F_i\circ \varphi$ is monomial in $t_1,\ldots,t_m$. Note that $m'\geqslant m$. If $m'>m$, there exists $F_l$ such that $F_l\circ \varphi$ is an alternating product of $F_i\circ \varphi$'s for $i\neq l$, which contradicts to the assumption that $F_l$ is irreducible.
	\end{proof}
\end{appendices}

\Addresses

\end{document}